\documentclass[12pt,reqno]{amsart}
\usepackage{amsfonts}

\usepackage{amssymb,amsmath,graphicx,amsfonts,euscript}
\usepackage{color}

\newtheorem{thm}{Theorem}[section]

\newtheorem{lemma}{Lemma}[section]

\theoremstyle{definition}
\newtheorem{define}{Definition}[section]

\theoremstyle{remark}
\newtheorem{rem}{Remark}[section]

\allowdisplaybreaks \voffset=-0.2in \numberwithin{equation}{section}

\begin{document}
\bigskip

\centerline{\Large\bf  Global regularity of 2D generalized incompressible}

\smallskip
\centerline{\Large\bf  magnetohydrodynamic  equations}

\bigskip

\centerline{Chao Deng$^{1}$, Zhuan Ye$^{1}$, Baoquan Yuan$^{2}$, Jiefeng Zhao$^{2}$}

\bigskip

\centerline{$^{1}$ Department of Mathematics and Statistics, Jiangsu Normal University, }
\medskip

\centerline{101 Shanghai Road, Xuzhou 221116, Jiangsu, PR China}

\medskip
\medskip
\centerline{$^{2}$ School of Mathematics and Information Science, Henan Polytechnic University,}
\medskip

\centerline{Jiaozuo 454003, PR China}

\medskip
\medskip
\centerline{E-mails: \texttt{dengxznu@gmail.com; yezhuan815@126.com
}}
\centerline{  \texttt{ bqyuan@hpu.edu.cn; zhaojiefeng003@hpu.edu.cn
}}

\bigskip
\bigskip
{\bf Abstract:}~~%
In this paper, we are concerned with the two-dimensional (2D) incompressible magnetohydrodynamic (MHD) equations with velocity dissipation given by $(-\Delta)^{\alpha}$ and magnetic diffusion given by reducing about logarithmic diffusion from standard Laplacian diffusion. More precisely, we establish the global regularity of solutions to the system as long as the power $\alpha$ is a positive constant. In addition, we prove several global \emph{a priori} bounds for the case $\alpha=0$. In particular, our results significantly improve previous
works and take us one step closer to a complete resolution of the global regularity issue on the 2D resistive MHD equations, namely, the case when the MHD equations only have standard Laplacian magnetic diffusion.

{\vskip 1mm
 {\bf AMS Subject Classification 2010:}\quad 35Q35; 35B65; 76W05; 76D03.

 }

\vskip .4in
\section{Introduction and main results}
The 2D generalized incompressible magnetohydrodynamic (GMHD) equations take the following form
\begin{equation}\label{ADGMHD}
\left\{\aligned
&\partial_{t}u+(u\cdot\nabla)u+\nu_{1}\mathcal{L}_{1}u+\nabla p=(b\cdot\nabla) b,\qquad x \in \mathbb{R}^{2},\, t>0,\\
&\partial_{t}b+(u\cdot\nabla)b+\nu_{2}\mathcal{L}_{2}b=(b\cdot\nabla)u,
\\
&\nabla\cdot u=0,\ \ \ \nabla\cdot b=0,\\
&u(x,0)=u_{0}(x),\,\,b(x,0)=b_{0}(x),
\endaligned \right.
\end{equation}
where $u=(u_{1}(x,t),\,u_{2}(x,t))$ denotes the velocity, $p=p(x,t)$ the scalar pressure and
$b=(b_{1}(x,t),\,b_{2}(x,t))$ the magnetic field of the fluid.
$u_{0}(x)$ and $b_{0}(x)$ are the given initial data satisfying $\nabla\cdot u_{0}=\nabla\cdot b_{0}=0$. Here $\nu_{1}$ and $\nu_{2}$ are nonnegative constants. The operators $\mathcal{L}_{i}$ ($i=1,2$) are Fourier multipliers with symbol $m_{i}(\xi)$, namely,
$$\widehat{\mathcal{L}_{i}f}(\xi)=m_{i}(\xi)\widehat{f}(\xi).$$
When $\nu_{1}>0,\,\nu_{2}>0$ and $m_{1}(\xi)=m_{2}(\xi)=|\xi|^{2}$, \eqref{ADGMHD} reduces to the classical MHD equations, which model the complex interaction between the fluid dynamic phenomena, such as the magnetic reconnection in astrophysics and geomagnetic dynamo in geophysics, plasmas, liquid metals, and salt water, etc (see, e.g., \cite{Davidson01,PF}).
The fundamental concept behind MHD is that magnetic fields can induce currents in a moving conductive fluid, which in turn creates forces on the fluid and also changes the magnetic field itself.

\vskip .1in
Due to the physical background and mathematical relevance, the MHD equations  attracted a lot of attention. The most fundamental problem concerning the MHD equations is whether physically relevant regular solutions remain smooth for all time or they develop finite time singularities with general initial data.
Actually, there have been substantial developments on the global regularity of the MHD equations with various form of the dissipation. Obviously, the classical 2D MHD equations have a unique global smooth solution (e.g. \cite{ST}). In the completely inviscid case, namely, $\nu_{1}=\nu_{2}=0$, the question of whether the local smooth solution develops singularity in finite time remains open. Therefore, it is interesting to consider the intermediate cases. Actually, one of the most interesting intermediate case reads $m_{1}(\xi)=|\xi|^{2\alpha},\,m_{2}(\xi)=|\xi|^{2\beta}$, which becomes the so-called 2D fractional MHD equations, namely,
\begin{equation}\label{2dGMHD}
\left\{\aligned
&\partial_{t}u+(u\cdot\nabla)u+(-\Delta)^{\alpha}u+\nabla p=(b\cdot\nabla) b,\qquad x \in \mathbb{R}^{2},\, t>0,\\
&\partial_{t}b+(u\cdot\nabla)b+(-\Delta)^{\beta}b=(b\cdot\nabla)u,
\\
&\nabla\cdot u=0,\ \ \ \nabla\cdot b=0,\\
&u(x,0)=u_{0}(x),\ \ b(x,0)=b_{0}(x),
\endaligned \right.
\end{equation}
where the fractional operator $(-\Delta)^{\gamma}$ is defined by the Fourier transform, namely
$$
\widehat{(-\Delta)^{\gamma} f}(\xi)=|\xi|^{2\gamma}\hat{f}(\xi).
$$
For simplicity, we denote  $\Lambda\triangleq\sqrt{-\Delta}$.
We remark the convention that by $\alpha=0$ we mean that there is no dissipation in $(\ref{2dGMHD})_{1}$, and similarly $\beta=0$ represents that there is no diffusion in $(\ref{2dGMHD})_{2}$. Recently, a number of works have been dedicated to the study of the global regularity of \eqref{2dGMHD} with general initial data (see e.g. \cite{CWYSiam14,FNZ14MM,JZ31114,JZ31115,TYZ113,TYZ,Wu2003,Wu2011,YB2014JMAA,Y3efg5,YX2014NA} and the references cited therein). To the best of our knowledge, the issue of the global regularity for \eqref{2dGMHD} with $\alpha=0,\beta=1$ is still a challenging open problem (see \cite{CW2011,LZ}). Interestingly, if more dissipation is added, then the corresponding system do admit a unique global regular solution. On the one hand, when $\alpha>0,\beta=1$, the global regularity issue was solved by Fan et al. \cite{FNZ14MM}, which was further improved by Yuan and Zhao \cite{YZhao16} logarithmically. Precisely, they obtained the global regularity of solutions requiring the dissipative operators weaker than any power of the fractional Laplacian. On the other hand, when $\alpha=0,\beta>1$, the global regularity of smooth solutions for the corresponding system was established in \cite{CWYSiam14,JZ31115} with different approaches, which was also improved by \cite{Agelas,Yejee18} logarithmically. Partially inspired by the result on logarithmically supercritical phenomena for Navier-Stokes equations \cite{TTao}, Wu \cite{Wu2011} proved that if $m_{1}(\xi)\geq \frac{|\xi|^{2\alpha}}{g_{1}(\xi)}$, $m_{2}(\xi)\geq \frac{|\xi|^{2\beta}}{g_{2}(\xi)}$ with $g_{1}\geq1, g_{2}\geq1$  are nondecreasing functions satisfying
$$\int_{1}^{\infty}\frac{1}{\tau[g_{1}(\tau)+g_{2}(\tau)]^{2}}\,d\tau=+\infty,$$
where $\alpha\geq1,\beta>0$ and $\alpha+\beta\geq2$, then \eqref{ADGMHD} has a unique global classical solution (see the further improved result on a periodic domain \cite{Yamazaki18}). Subsequently, Tran et al. \cite{TYZ} considered the case $\alpha\geq2,\,\beta=0$ and proved that if $m_{1}(\xi)\geq \frac{|\xi|^{2\alpha}}{g_{1}(\xi)}$ and $g_{1}\geq1$ is a radial function satisfying $g_{1}(\tau)\leq C\ln (e+\tau)$, then the solution is global regular.
This result was later improved by Yamazaki \cite{Yasd14a} only requiring
$$\int_{e}^{\infty}\frac{1}{\tau \ln\tau g_{1}(\tau)}\,d\tau=+\infty.$$

\vskip .1in
Summarizing the results mentioned above, the direct energy method tells us that to get an $H^{1}$-bound for the solution, we just only have Laplacian magnetic diffusion $\alpha=0,\beta\geq1$ or we ask for $\alpha+\beta\geq2,\,\beta<1$ (see \cite{Wu2003,Wu2011,TYZ} for details). If $\beta<1$, the necessary power $\alpha$ jumps from $0$ to $2-\beta$ immediately! This indicates that it is extremely difficult to establish the global regularity of \eqref{2dGMHD} with $\alpha+\beta<2$ and $\beta<1$. Actually, up to date, the global existence and uniqueness of regular solution remain completely open for this case $\alpha+\beta<2$ and $\beta<1$ with large initial data. Just as Tao \cite{TTao}, since it is very difficult to obtain the global regularity for this case, it is possible to achieve this goal when the magnetic diffusion is some logarithmically weaker than a full Laplacian. In fact, we are able to show that if the magnetic diffusion is some logarithmically weaker than a full Laplacian, then the global $H^{1}$-bound of the solution can be achieved as long as the fractional dissipation power of the velocity field is positive. More precisely, this paper is devoted to the global regularity problem of the following 2D incompressible GMHD equations
\begin{equation}\label{GMHD}
\left\{\aligned
&\partial_{t}u+(u\cdot\nabla)u+(-\Delta)^{\alpha}u+\nabla p=(b\cdot\nabla) b,\qquad x \in \mathbb{R}^{2},\, t>0,\\
&\partial_{t}b+(u\cdot\nabla)b+\mathcal{L}b=(b\cdot\nabla)u,
\\
&\nabla\cdot u=0,\ \ \ \nabla\cdot b=0,\\
&u(x,0)=u_{0}(x),\,\,b(x,0)=b_{0}(x),
\endaligned \right.
\end{equation}
where the operator $\mathcal{L}$ is a Fourier multiplier with symbol $\frac{|\xi|^{2}}{g(\xi)}$, namely,
$$\widehat{\mathcal{L}b}(\xi)=\frac{|\xi|^{2}}{g(\xi)}\widehat{b}(\xi),$$
with $g(\xi)=g(|\xi|)$ a radial non-decreasing smooth function satisfying the following two conditions

(a)\,\, $g$ obeys that
$$g(\xi)\geq C_{0}>0 \quad \mbox{for all}\ \ \xi\geq0;$$

(b)\,\, $g$ is of the Mikhlin-H$\rm\ddot{o}$mander type, namely, a constant $\widetilde{C}>0$ such that
$$|\partial_{\xi}^{l}g(\xi)|\leq \widetilde{C}|\xi|^{-l}|g(\xi)|,\quad l\in \{1,\,2 \},\  \forall\,\xi\neq0.$$

Here we point out that Zhao \cite{Zhao22} recently proved the global regularity result for \eqref{GMHD} as long as $\alpha>\frac{1}{4}$ and $g$ fulfills
$$g(\xi)\approx \ln\ln(\sigma+|\xi|)\ln\ln\ln(\sigma+|\xi|)\cdot\cdot\cdot
\underbrace{\ln\ln\cdot\cdot\cdot\ln}_{k\,\rm{times}}(\sigma+|\xi|)$$
with $\sigma \geq \underbrace{\exp \exp \cdots \exp}_{k\,\rm{times}}\triangleq \sigma(k)$. Here and in what follows, $f\approx g$ means $C_{1}g \leq f\leq C_{2} g$ for some absolute constant $C_{2}\geq C_{1}>0$.
As the restrictions in \cite{Zhao22} are somewhat unnatural, in particular the condition $\alpha>\frac{1}{4}$, the main goal of this paper is to weaken the conditions required in \cite{Zhao22} from two aspects: $\alpha$ and $g$, but the same global regularity result is still guaranteed. More precisely, our main result reads as follows.
\begin{thm}\label{Th1} Let $(u_{0}, b_{0})
\in H^{s}(\mathbb{R}^{2})\times H^{s}(\mathbb{R}^{2})$ with $s\geq2$ and satisfy $\nabla\cdot u_{0}=\nabla\cdot b_{0}=0$. Assume that $g$ satisfies (a)-(b) and
\begin{equation}\label{logcd1}
g(\xi)\leq \widetilde{C}\left(\ln(\sigma+|\xi|)\ln\ln(\sigma+|\xi|)\cdot\cdot\cdot
\underbrace{\ln\ln\cdot\cdot\cdot\ln}_{k\,\rm{times}}(\sigma+|\xi|)\right)^{\frac{1}{2}},
\end{equation}
where the constant $\widetilde{C}>0$ and $\sigma \geq \sigma(k)$.
If $\alpha>0$, then (\ref{GMHD}) admits a unique global solution such that for any given $T>0$,
$$u,\,\,b\in L^{\infty}\big([0, T]; H^{s}(\mathbb{R}^{2})\big),\qquad \Lambda^{\alpha}u,\,\, \mathcal{L}^{\frac{1}{2}}b\in L^{2}\big([0, T]; H^{s}(\mathbb{R}^{2})\big).$$
\end{thm}

\begin{rem}\rm
Compared with \cite{Zhao22}, we are not only relaxing the assumption of $g$ from double logarithm to logarithm, but also relaxing $\alpha>\frac{1}{4}$ to $\alpha>0$. Moreover, in the special case when $g$ is a positive constant, Theorem \ref{Th1} re-establishes the global
regularity for \eqref{2dGMHD} with $\alpha>0,\beta=1$.
\end{rem}
\begin{rem}\rm
On the one hand, it is interesting to consider the problem whether Theorem \ref{Th1} is still valid when $(-\Delta)^{\alpha}$ is replaced by some logarithmic operators, such as $[\ln (e+\Lambda)]^{\rho}$ for some $\rho>0$.
On the other hand, at the moment it is also not clear whether Theorem \ref{Th1} holds true when the power one half of \eqref{logcd1} is replaced by $\tau$ with $\tau>\frac{1}{2}$. The only place where we require $\tau\leq\frac{1}{2}$ lies in proving Lemma \ref{tL32} below. We will investigate these problems further in our future work.
\end{rem}

\begin{rem}\rm
It follows from the proof of Lemma \ref{tL32} below that the unique global existence classical solution in Sobolev spaces has at least a $(k+1)$-multiple exponential upper bound uniformly in times, which particularly implies that it is very difficult to obtain the global regular solution even for reducing logarithm-type diffusion for magnetic field. Furthermore, this also indicates that the well-known global regularity problem on the 2D resistive MHD equations is a critical open problem.
\end{rem}

\vskip .1in
Finally, we do not have a complete solution for the case $\alpha=0$, but we are able to establish several global \emph{a priori} bounds for this case, which may be helpful in the eventual resolution of the global regularity problem for the 2D resistive MHD equations.

\begin{thm}\label{Th2} Consider (\ref{GMHD}) with $\alpha=0$, namely,
\begin{equation}\label{ZEGMHD}
\left\{\aligned
&\partial_{t}u+(u\cdot\nabla)u+\nabla p=(b\cdot\nabla) b,\qquad x \in \mathbb{R}^{2},\, t>0,\\
&\partial_{t}b+(u\cdot\nabla)b+\mathcal{L}b=(b\cdot\nabla)u,
\\
&\nabla\cdot u=0,\ \ \ \nabla\cdot b=0,\\
&u(x,0)=u_{0}(x),\,\,b(x,0)=b_{0}(x).
\endaligned \right.
\end{equation}
Let $(u_{0}, b_{0})
\in H^{s}(\mathbb{R}^{2})\times H^{s}(\mathbb{R}^{2})$ with $s\geq2$ and satisfy $\nabla\cdot u_{0}=\nabla\cdot b_{0}=0$. If $g$ satisfies the assumptions stated in Theorem \ref{Th1}, then the solution $(u,b)$ of \eqref{ZEGMHD} admits the following global bounds for any $r\in[0,1)$
$$\|\omega(t)\|_{L^{2}}+\|\Lambda^{r}j(t)\|_{L^{2}}+\|\mathcal{L}b(t)\|_{L^{2}} \leq C(t,u_{0},b_{0}).$$
In addition, if $\Lambda^{r}\{\mathcal{L}b_{0}-
(b_{0}\cdot\nabla)u_{0}\}\in L^{2}(\mathbb{R}^{2})$, then it also holds
$$\|\Lambda^{r}\{\mathcal{L}b-
(b\cdot\nabla)u\}(t)\|_{L^{2}}\leq C(t,u_{0},b_{0}).$$
\end{thm}

\begin{rem}\rm
It follows from the global \emph{a priori} bound of Lemma \ref{tL32} that \eqref{ZEGMHD} indeed admits a global $H^1$-weak solution. Although we have derived several global \emph{a priori} bounds in Theorem \ref{Th2}, it is not clear if such weak solutions are unique or can be improved to global classical solutions.
\end{rem}

\vskip .2in
The rest of this paper is divided into two sections and an appendix. In Section \ref{sectt2}, we collect some properties of the operator $\mathcal{L}$ and the definition of Besov spaces as well as several facts, which will be used in this paper. Section \ref{sectt3} is devoted to the proof of Theorem \ref{Th1}. Finally, the proof of Theorem \ref{Th2} is presented in Section \ref{addedxq}.

\vskip .1in
\textbf{Notation}: Throughout this paper, the letter $C$ denotes various positive and finite constants whose exact values are unimportant and may vary from line to line.
To emphasize the dependence of a constant on some certain quantities $\rho_{1},\rho_{2},\cdot\cdot\cdot$, we write $C(\rho_{1},\rho_{2},\cdot\cdot\cdot)$. Moreover, throughout this paper, $C(t,\cdot\cdot\cdot)$ is nondecreasing in terms of the time variable, and is bounded as $t$ approaches to zero.
Let $\mathbb{X}$ be a Banach space. For $p\in[1,\infty]$, the notation $L^p(0,T;\mathbb{X})$ or $L_{T}^p(\mathbb{X})$ stands for the set of measurable functions on $(0,T)$ with values in $\mathbb{X}$, such that $t\mapsto \|h(t)\|_{\mathbb{X}}$ belongs to $L^p(0,T)$.

\vskip .3in
\section{Preliminaries}\setcounter{equation}{0}\label{sectt2}
\subsection{Several properties of the operator $\mathcal{L}$} This subsection is mainly devoted to establishing some properties for the operator $\mathcal{L}$, which play some important roles in proving our main theorem. We mention that the estimates of this subsection is partially inspired by our previous paper \cite{Yejee18}.
To begin with, we consider the linear inhomogeneous equation
\begin{equation}\label{linear}
\left\{\aligned
&\partial_{t}W+\mathcal{L}W=f,
\\
&W(x,0)=W_{0}(x).
\endaligned \right.
\end{equation}
Thanks to the Fourier transform method, the solution of the linear inhomogeneous equation (\ref{linear}) can be explicitly given by
\begin{equation}
W(t)=K(t)\ast W_{0}+\int_{0}^{t}{K(t-\tau)\ast f(\tau)\,d\tau},\nonumber
\end{equation}
where the kernel $K$ satisfies
\begin{eqnarray}
K(x,t)=\mathcal{F}^{-1}
\Big(e^{\frac{-t|\xi|^{2}}{g(\xi)}}\Big)(x).\nonumber
\end{eqnarray}
Here and in what follows, $\mathcal{F}^{-1}{g}$ denotes the inverse Fourier transform of $g$. We also denote $\mathcal{F}{g}$ as the Fourier transform of $g$.

\vskip .1in
Now we will establish the following crucial estimate.
\begin{lemma}\label{tL31}
For any $k\geq0$ and $s>k-1$, there exists a constant $C$ depending only on $s$ and $k$ such that for any $t>0$
\begin{eqnarray}\label{tok202}
\int_{\mathbb{R}^{2}}{\frac{|\xi|^{2s}}{g^{k}(\xi)}e^{-\frac{2t|\xi|^{2}}{g(\xi)}}
\,d\xi}\leq Ct^{-(s+1)} g(A_{t})^{s-k+1},
\end{eqnarray}
where $A_{t}>0$ is the unique solution of the following equation
\begin{eqnarray}
\label{fgetfwf}
\frac{g(x)}{x^{2}}=t.
\end{eqnarray}
\end{lemma}
\begin{proof}
To begin with, we split the integral into two parts
\begin{eqnarray}
\int_{\mathbb{R}^{2}}{\frac{|\xi|^{2s}}{g^{k}(\xi)}e^{-\frac{2t|\xi|^{2}}{g(\xi)}}
\,d\xi}
=\int_{|\xi|\leq R}{\frac{|\xi|^{2s}}{g^{k}(\xi)}e^{-\frac{2t|\xi|^{2}}{g(\xi)}}\,d\xi}
+\int_{|\xi|\geq R}{\frac{|\xi|^{2s}}{g^{k}(\xi)}e^{-\frac{2t|\xi|^{2}}{g(\xi)}}\,d\xi},\nonumber
\end{eqnarray}
where $R>0$ will be fixed later. Thanks to $s>k-1$, we have
\begin{align}\label{tok203}
\int_{|\xi|\leq R}{\frac{|\xi|^{2s}}{g^{k}(\xi)}e^{-\frac{2t|\xi|^{2}}{g(\xi)}}\,d\xi} =&(2t)^{-k}\int_{|\xi|\leq R}{|\xi|^{2(s-k)} \left(\frac{2t|\xi|^{2}}{g(\xi)}\right)^{k}
e^{-\frac{2t|\xi|^{2}}{g(\xi)}}\,d\xi}\nonumber\\ \leq&Ct^{-k}\int_{|\xi|\leq R}{|\xi|^{2s-2k}  \,d\xi}\nonumber\\ \leq&Ct^{-k}\int_{0}^{R}{r^{2s-2k+1}  \,dr}\nonumber\\ \leq&Ct^{-k}R^{2(s-k+1)},
\end{align}
where we have used the simple fact
\begin{eqnarray}\label{tok204}
\max_{\lambda\geq0}(\lambda^{k}e^{-\lambda})\leq C(k).
\end{eqnarray}
Due to the fact that $g$ is a logarithmic function like \eqref{logcd1}, it is not hard to check that for any $\epsilon>0$, there exists $C(\epsilon)$ such that (see the end for its proof)
\begin{eqnarray}\label{chkhgtt1}
\frac{r^{\epsilon}}{g(r)}\geq C(\epsilon)\frac{R^{\epsilon}}{g(R)},\ \ \forall\,r\geq R.
\end{eqnarray}
Thanks to (\ref{tok204}) and \eqref{chkhgtt1}, one obtains
\begin{align}\label{tok205}
\int_{|\xi|\geq R}{\frac{|\xi|^{2s}}{g^{k}(\xi)}e^{-\frac{2t|\xi|^{2}}{g(\xi)}}\,d\xi} =&(2t)^{-k}\int_{|\xi|\geq R}{|\xi|^{2(s-k)} \left(\frac{2t|\xi|^{2}}{g(\xi)}\right)^{k}
e^{-\frac{t|\xi|^{2}}{g(\xi)}}e^{-\frac{t|\xi|^{2}}{g(\xi)}}\,d\xi}\nonumber\\ \leq&
Ct^{-k}
\int_{|\xi|\geq R}{|\xi|^{2s-2k}e^{-\frac{t|\xi|^{2}}{g(\xi)}}  \,d\xi}
\nonumber\\ =&
Ct^{-k}
\int_{R}^{+\infty}{r^{2s-2k+1}e^{-\frac{tr^{2}}{g(r)}}  \,dr}
\nonumber\\ =&
Ct^{-k}
\int_{R}^{+\infty}{r^{2s-2k+1}e^{-tr^{2-\epsilon}\frac{r^{\epsilon}}{g(r)}} \,dr}
\nonumber\\ \leq&
Ct^{-k}
\int_{R}^{+\infty}{r^{2s-2k+1}e^{-C_{\epsilon}tr^{2-\epsilon}\frac{R^{\epsilon}}{g(R)}} \,dr}\nonumber\\ \leq&Ct^{-\frac{2(s+1)-k\epsilon}{2-\epsilon}}
R^{-\frac{2(s-k+1)\epsilon}{2-\epsilon}}
g(R)^{\frac{2(s-k+1)}{2-\epsilon}}.
\end{align}
Combining (\ref{tok203}) and (\ref{tok205}) yields
\begin{align}
\int_{\mathbb{R}^{2}}{\frac{|\xi|^{2s}}{g^{k}(\xi)}e^{-\frac{2t|\xi|^{2}}{g(\xi)}}\,d\xi}
 \leq& Ct^{-k}R^{2(s-k+1)}+Ct^{-\frac{2(s+1)-k\epsilon}{2-\epsilon}}
R^{-\frac{2(s-k+1)\epsilon}{2-\epsilon}}
g(R)^{\frac{2(s-k+1)}{2-\epsilon}}.\nonumber
\end{align}
Now we fix $R$ such that
$$t^{-k}R^{2(s-k+1)}=t^{-\frac{2(s+1)-k\epsilon}{2-\epsilon}}
R^{-\frac{2(s-k+1)\epsilon}{2-\epsilon}}
g(R)^{\frac{2(s-k+1)}{2-\epsilon}}$$
or equivalently
$$\frac{g(R)}{R^{2}}=t.$$
Then we deduce
\begin{eqnarray}
\int_{\mathbb{R}^{2}}{\frac{|\xi|^{2s}}{g^{k}(\xi)}e^{-\frac{2t|\xi|^{2}}{g(\xi)}}\,d\xi}
\leq 2Ct^{-(s+1)} g(R)^{s-k+1}=2Ct^{-(s+1)} g(A_{t})^{s-k+1}
,\nonumber
\end{eqnarray}
which is \eqref{tok202}.
Finally, we say some words about \eqref{chkhgtt1}.
Actually, we rewrite \eqref{chkhgtt1} as
\begin{eqnarray}
\left(\frac{r}{R}\right)^{\epsilon}\geq C(\epsilon)\frac{g(r)}{g(R)},\ \ \forall\,r\geq R.\nonumber
\end{eqnarray}
Recalling the condition (a) and the fact $g$ is a logarithmic function like \eqref{logcd1}, we may derive
\begin{equation}
 \left\{\aligned
&g(r)\leq g\left(R^{2}\right)\leq C g\left(R\right)\leq \frac{C}{C_{0}} g\left(\frac{r}{R} \right) g\left(R\right), \ \quad\qquad \mbox{if}\ \  r\leq R^{2}\\
&g(r)\leq g\left(\frac{r^{2}}{R^{2}}\right)\leq C g\left(\frac{r}{R}\right)\leq \frac{C}{C_{0}} g\left(\frac{r}{R} \right) g\left(R\right), \qquad \mbox{if}\ \ r\geq R^{2},\nonumber
\endaligned \right.
\end{equation}
which yields
$$\frac{g(r)}{g(R)}\leq \frac{C}{C_{0}} g\left(\frac{r}{R} \right).$$
Thanks to $\frac{r}{R}\geq1$ and a logarithmic type function $g$, we conclude
$$ g\left(\frac{r}{R} \right)\leq C_{1}(\epsilon)\left(\frac{r}{R}\right)^{\epsilon}.$$
One thus has
$$\frac{g(r)}{g(R)}\leq \frac{C}{C_{0}}C_{1}(\epsilon)\left(\frac{r}{R}\right)^{\epsilon},$$
which of course ensures \eqref{chkhgtt1} by taking $C(\epsilon)=\frac{C_{0}}{C C_{1}(\epsilon)}$.
As a result, this completes the proof of Lemma \ref{tL31}.
\end{proof}

\vskip .1in
The following lemma is a consequence of Lemma \ref{tL31}.
\begin{lemma}\label{L32}
Let $s>-1$, then there exists a constant $C$ depending only on $s$ such that for any $t>0$
\begin{eqnarray}\label{tok206}
\|K(t)\|_{\dot{H}^{s}}\leq Ct^{-\frac{s+1}{2}}g(A_{t})^{\frac{s+1}{2}},
\end{eqnarray}
where $A_{t}$ is given by (\ref{fgetfwf}).
\end{lemma}
\begin{proof}
According to (\ref{tok202}), one obtains
\begin{align}
\|K(t)\|_{\dot{H}^{s}}^{2} =&
\int_{\mathbb{R}^{2}}{|\xi|^{2s}|\widehat{K}(\xi,t)|^{2}} \,d\xi\nonumber\\
 =&
\int_{\mathbb{R}^{2}}{|\xi|^{2s}e^{-\frac{2t|\xi|^{2}}{g(\xi)}}} \,d\xi
\nonumber\\ \leq& Ct^{-(s+1)}g(A_{t})
^{s+1}.\nonumber
\end{align}
Therefore, we conclude the proof of Lemma \ref{L32}.
\end{proof}

\vskip .1in
Next we would like to show the following lemma.
\begin{lemma}\label{tL33}
 There exists a constant $C$ such that for any $t>0$
\begin{eqnarray}\label{tok208}
\|\Lambda^{2-\delta}K(t)\|_{L^{1}}\leq Ct^{-(1-\frac{\delta}{2})}g^{1-\frac{\delta}{2}}(A_{t}),
\end{eqnarray}
where $\delta\in[0,1)$ and $A_{t}$ is given by (\ref{fgetfwf}).
\end{lemma}
\begin{proof}
We first notice that for any radial function $g$ in $\mathbb{R}^{2}$, there holds
$$\widehat{g}(x)=(\mathcal{F}^{-1}g)(x).$$
Obviously, $\widehat{K}(\xi,t)$ is a radial function and so does $|\xi|^{2-\delta}\widehat{K}(\xi,t)$.
By means of this observation, it yields that
\begin{align}\label{tok209}
\|\Lambda^{2-\delta}K(t)\|_{L^{1}}&= \|\mathcal{F}^{-1}(|\xi|^{2-\delta}\widehat{K}(\xi,t))\|_{L^{1}}
\nonumber\\&= \|\mathcal{F}(|\xi|^{2-\delta}\widehat{K}(\xi,t))\|_{L^{1}}
\nonumber\\&\leq C \||\xi|^{2-\delta}\widehat{K}(\xi,t)\|_{L^{2}}^{\frac{1}{2}}
\|\nabla_{\xi}^{2}(|\xi|^{2-\delta}\widehat{K}(\xi,t))\|_{L^{2}}^{\frac{1}{2}},
\end{align}
where in the last line we have applied the interpolation inequality
\begin{eqnarray} \|\widehat{h}\|_{L^{1}}\leq C \|h\|_{L^{2}}^{\frac{1}{2}}\|\nabla^{2}h\|_{L^{2}}^{\frac{1}{2}}.\nonumber
\end{eqnarray}
The above interpolation inequality is an easy consequence of the high-low frequency decomposition. More precisely, we have for any $\rho>1$ that
\begin{align}
\|\widehat{f}\|_{L^{1}}&= \int_{\mathbb{R}^{2}} |\widehat{f}(\xi)|\,d\xi\nonumber\\
&= \int_{|\xi|\leq \aleph} |\widehat{h}(\xi)|\,d\xi+\int_{|\xi|\geq \aleph} |\widehat{f}(\xi)|\,d\xi
\nonumber\\
&= \int_{|\xi|\leq \aleph} |\widehat{f}(\xi)|\,d\xi+\int_{|\xi|\geq \aleph}|\xi|^{-\rho}|\xi|^{\rho} |\widehat{f}(\xi)|\,d\xi
\nonumber\\
&\leq C N\Big(\int_{|\xi|\leq \aleph} |\widehat{f}(\xi)|^{2}\,d\xi\Big)^{\frac{1}{2}}+C\aleph^{1-\rho}\Big(\int_{|\xi|\geq \aleph}|\xi|^{2\rho} |\widehat{f}(\xi)|^{2}\,d\xi\Big)^{\frac{1}{2}}
\nonumber\\
&\leq C \aleph\|f\|_{L^{2}}+C\aleph^{1-\rho}\|\Lambda^{\rho}f\|_{L^{2}}
\nonumber\\
&\leq C
\|f\|_{L^{2}}^{1-\frac{1}{\rho}}\|\Lambda^{\rho}f\|_{L^{2}}^{\frac{1}{\rho}},\nonumber
 \end{align}
where in the last line we have fixed $\aleph$ as
$$\aleph=\left(\frac{\|\Lambda^{\rho}f\|_{L^{2}}}{\|f\|_{L^{2}}}\right)^{\frac{1}{\rho}}.$$
Using (\ref{tok206}), we have
\begin{eqnarray}\label{tok210}
\||\xi|^{2-\delta}\widehat{K}(\xi,t)\|_{L^{2}}
\leq Ct^{-\frac{3-\delta}{2}}g^{\frac{3-\delta}{2}}(A_{t}).
\end{eqnarray}
It follows from some direct computations that
$$|\nabla_{\xi}^{2}(|\xi|^{2-\delta}\widehat{K}(\xi,t))|\leq C(|\xi|^{-\delta}|\widehat{K}(\xi,t)|+|\xi|^{1-\delta}\,|\nabla_{\xi}\widehat{K}(\xi,t)|
+|\xi|^{2-\delta}|\nabla_{\xi}^{2}\widehat{K}(\xi,t)|).$$
Recalling $\widehat{K}(\xi,t)=e^{-\frac{t|\xi|^{2}}{g(\xi)}}$ and the property $|\partial_{\xi}^{k}g(\xi)|\leq \widetilde{C}|\xi|^{-k}|g(\xi)|,\, k\in \{1,\,2\}$, straightforward computations lead to the following estimates
\begin{align}
|\nabla_{\xi}\widehat{K}(\xi,t)|
 &\leq  C\left(\frac{t|\xi|}{g(\xi)}+\frac{t|\xi|^{2}}{g^{2}(\xi)}|\partial_{\xi}g(\xi)|\right)|\widehat{K}(\xi,t)|
\nonumber\\&\leq   C\frac{t|\xi||\widehat{K}(\xi,t)|}{g(\xi)},\nonumber
\end{align}
\begin{align}
|\nabla_{\xi}^{2}\widehat{K}(\xi,t)|
\leq & C\Big(\frac{t^{2}|\xi|^{2}}{g^{2}(\xi)}+\frac{t^{2}|\xi|^{4}|\partial_{\xi}g(\xi)|^{2}}{
g^{4}(\xi)}+\frac{t}{g(\xi)} +\frac{t|\xi|^{2}|\partial_{\xi}g(\xi)|^{2}}{g^{3}(\xi)}\nonumber\\&+
\frac{t|\xi||\partial_{\xi}g(\xi)|}{g^{2}(\xi)}
+\frac{t|\xi|^{2}|\partial_{\xi}^{2}g(\xi)|}{g^{2}(\xi)}\Big)
 |\widehat{K}(\xi,t)|\nonumber\\
\leq &C\Big(\frac{t^{2}|\xi|^{2}}{g^{2}(\xi)}+\frac{t}{g(\xi)}\Big)|\widehat{K}(\xi,t)|.
\nonumber
\end{align}
Therefore, we obtain
$$|\nabla_{\xi}^{2}(|\xi|^{2-\delta}\widehat{K}(\xi,t))|\leq C\left(|\xi|^{-\delta}+\frac{t|\xi|^{2-\delta}}{g(\xi)}
+\frac{t^{2}|\xi|^{4-\delta}}{g^{2}(\xi)}\right)
|\widehat{K}(\xi,t)|.$$
In view of (\ref{tok202}) again, one concludes
\begin{align}\label{tok211}
\|\nabla_{\xi}^{2}(|\xi|^{2-\delta}\widehat{K}(\xi,t))\|_{L^{2}}^{2}&\leq
C\int_{\mathbb{R}^{2}}{|\xi|^{-2\delta}e^{-\frac{2t|\xi|^{2}}{g(\xi)}}\,d\xi}+C
t^{2}\int_{\mathbb{R}^{2}}{\frac{|\xi|^{4-2\delta}}{g^{2}(\xi)}e^{-\frac{2t|\xi|^{2}}
{g(\xi)}}\,d\xi}
\nonumber\\& \quad+C
t^{4}\int_{\mathbb{R}^{2}}{\frac{|\xi|^{8-2\delta}}{g^{4}(\xi)}
e^{-\frac{2t|\xi|^{2}}{g(\xi)}}\,d\xi}
\nonumber\\
&\leq
Ct^{-(1-\delta)}g^{1-\delta}(A_{t})+
Ct^{2}t^{-(3-\delta)}g^{1-\delta}(A_{t})
+C
t^{4}t^{-(5-\delta)}g^{1-\delta}(A_{t})\nonumber\\
&\leq
Ct^{-(1-\delta)}g^{1-\delta}(A_{t}),
\end{align}
where $\delta\in[0,1)$.
Inserting (\ref{tok210}) and (\ref{tok211}) into (\ref{tok209}), we finally gather
$$\|\Lambda^{2-\delta}K(t)\|_{L^{1}}\leq Ct^{-(1-\frac{\delta}{2})}g^{1-\frac{\delta}{2}}(A_{t})
.$$
We thus complete the proof of Lemma \ref{tL33}.
\end{proof}

\vskip .1in

Now we are ready to show the following key lemma.
\begin{lemma}\label{khrdlyvhbg}
Under the assumptions of Theorem \ref{Th1}, there holds for $\delta\in(0,1)$
\begin{align}\label{bnmper8}
\int_{0}^{T}{
\| K(t)\|_{L^{2}}^{2-\delta}\,dt}+\int_{0}^{T}{
\|\Lambda^{1-\delta} K(t)\|_{L^{2}}\,dt}+\int_{0}^{T}{
\|\Lambda^{2-\delta} K(t)\|_{L^{1}}\,dt}\leq C(T).
\end{align}
\end{lemma}

\begin{proof}
By \eqref{tok206} and \eqref{tok208}, it follows that
$$
\| K(t)\|_{L^{2}}^{2-\delta}+\|\Lambda^{1-\delta}  K(t)\|_{L^{2}}+
\|\Lambda^{2-\delta}  K(t)\|_{L^{1}}\leq Ct^{-(1-\frac{\delta}{2})}g^{1-\frac{\delta}{2}}(R),$$
where $R$ satisfies
$$t=\frac{g(R)}{R^{2}}.$$
Obviously, one gets
$$dt=\left(-2R^{-3}g(R)+R^{-2}g'(R)\right)dR,$$
which yields
\begin{align}
\int_{0}^{T}{t^{-(1-\frac{\delta}{2})}g^{1-\frac{\delta}{2}}(R)\,dt}&= \int_{+\infty}^{A_{T}}{R^{2-\delta}\left(-2R^{-3}g(R)+R^{-2}g'(R)\right)\, dR}\nonumber\\&= \int_{A_{T}}^{+\infty}{R^{2-\delta}\left(2R^{-3}g(R)-R^{-2}g'(R)\right)\, dR}
\nonumber\\&= 2\int_{A_{T}}^{+\infty}{R^{-1-\delta}g(R) \, dR}- \int_{A_{T}}^{+\infty}{R^{-\delta}g'(R) \, dR} \nonumber\\&\leq  C(T),\nonumber
\end{align}
where in the last line we have used (b) and the fact that $g$ is a logarithmic function satisfying \eqref{logcd1}. Consequently, we achieve
$$
\int_{0}^{T}{
\| K(t)\|_{L^{2}}^{2-\delta}\,dt}+\int_{0}^{T}{
\|\Lambda^{1-\delta} K(t)\|_{L^{2}}\,dt}+\int_{0}^{T}{
\|\Lambda^{2-\delta} K(t)\|_{L^{1}}\,dt}\leq C(T).
$$
This concludes the proof of Lemma \ref{khrdlyvhbg}.
\end{proof}

\vskip .1in

\subsection{Besov spaces and some useful facts}
This subsection provides the definition of Besov spaces and
several useful facts. Let us recall briefly the definition of the
Littlewood-Paley decomposition (see \cite{BCD} for more details). More precisely, we take some smooth radial non increasing function $\chi$ with values in $[0, 1]$ such that $\chi\in C_{0}^{\infty}(\mathbb{R}^{2})$ supported in the ball $\mathcal{B}\triangleq\{\xi\in \mathbb{R}^{2}, |\xi|\leq \frac{4}{3}\}$
and with value $1$ on $\{\xi\in \mathbb{R}^{2}, |\xi|\leq \frac{3}{4}\}$, then we define
$\varphi(\xi)\triangleq\chi\big(\frac{\xi}{2}\big)-\chi(\xi)$. It is clear that
${\varphi\in C_{0}^{\infty}(\mathbb{R}^{2})}$ is supported in the annulus
$\mathcal{C}\triangleq\{\xi\in \mathbb{R}^{2}, \frac{3}{4}\leq |\xi|\leq
\frac{8}{3}\}$ and satisfies
$$\chi(\xi)+\sum_{j\geq0}\varphi(2^{-j}\xi)=1, \quad  \forall \xi\in \mathbb{R}^{2};\qquad \sum_{j\in \mathbb{Z}}\varphi(2^{-j}\xi)=1, \quad  \forall \xi\neq0.$$
Let $h=\mathcal{F}^{-1}(\varphi)$ and $\widetilde{h}=\mathcal{F}^{-1}(\chi)$, then
the inhomogeneous dyadic blocks $\Delta_{j}$ of our decomposition by setting
$$\Delta_{j}u=0,\ \ j\leq -2; \ \  \ \ \ \Delta_{-1}u=\chi(D)u=\int_{\mathbb{R}^{2}}{\widetilde{h}(y)u(x-y)\,dy};
$$
$$ \Delta_{j}u=\varphi(2^{-j}D)u=2^{jn}\int_{\mathbb{R}^{2}}{h(2^{j}y)u(x-y)\,dy},\ \ \forall j\geq0.
$$
The homogeneous dyadic blocks $\dot{\Delta}_{j}$ reads
$$\dot{\Delta}_{j}u=\varphi(2^{-j}D)u =2^{2j}\int_{\mathbb{R}^{2}}{h(2^{j}y)u(x-y)\,dy},\ \ \forall j\in \mathbb{Z}.$$

\vskip .1in
It is now time to introduce the inhomogeneous Besov spaces, which are defined by the Littlewood-Paley decomposition.
\begin{define}
Let $s\in \mathbb{R}, (p,r)\in[1,+\infty]^{2}$. The inhomogeneous
Besov space $B_{p,r}^{s}$ is defined as a space of $f\in
S'(\mathbb{R}^{2})$ such that
$$ B_{p,r}^{s}=\{f\in S'(\mathbb{R}^{2});  \|f\|_{B_{p,r}^{s}}<\infty\},$$
where
\begin{equation} \nonumber
 \|f\|_{B_{p,r}^{s}}\triangleq\left\{\aligned
&\Big(\sum_{j\geq-1}2^{jrs}\|\Delta_{j}f\|_{L^{p}}^{r}\Big)^{\frac{1}{r}}, \quad \ r<\infty,\\
&\sup_{j\geq-1}
2^{js}\|\Delta_{j}f\|_{L^{p}}, \quad \ r=\infty.\\
\endaligned\right.
\end{equation}
\end{define}

We denote the tempered distributions by $S'(\mathbb{R}^{2})$, and polynomials by $\mathcal{P}(\mathbb{R}^{2})$. Now the homogeneous Besov spaces are defined as follows.
\begin{define}
Let $s\in \mathbb{R}, (p,r)\in[1,+\infty]^{2}$. The homogeneous
Besov space $\dot{B}_{p,r}^{s}$ is defined as a space of $f\in
S'(\mathbb{R}^{2})/\mathcal{P}(\mathbb{R}^{2})$ such that
$$ \dot{B}_{p,r}^{s}=\{f\in S'(\mathbb{R}^{2})/\mathcal{P}(\mathbb{R}^{2});  \|f\|_{\dot{B}_{p,r}^{s}}<\infty\},$$
where
\begin{equation}\nonumber
 \|f\|_{\dot{B}_{p,r}^{s}}\triangleq\left\{\aligned
&\Big(\sum_{j\in \mathbb{Z}}2^{jrs}\|\dot{\Delta}_{j}f\|_{L^{p}}^{r}\Big)^{\frac{1}{r}}, \quad \ 1\leq r<\infty,\\
&\sup_{j\in \mathbb{Z}}
2^{js}\|\dot{\Delta}_{j}f\|_{L^{p}}, \quad \ r=\infty.\\
\endaligned\right.
\end{equation}
\end{define}

\vskip .1in
The following one is the classical Bernstein type inequality (see \cite[Lemma 2.1]{BCD}).
\begin{lemma} \label{vfgty8xc}
Let $1\leq a\leq b\leq\infty$ and $\mathcal{C}$ be an annulus and $\mathcal{B}$ a ball of $\mathbb{R}^{2}$, then
$${\rm Supp}\widehat{f}\subset \lambda \mathcal{B}\ \Rightarrow\
\|\nabla^{k}f\|_{L^b}\leq C_1\, \lambda^{k  +
2(\frac{1}{a}-\frac{1}{b})} \|f\|_{L^a},
$$
$$  {\rm Supp}\widehat{f}\subset \lambda \mathcal{C}\ \Rightarrow\
C_2\, \lambda^{k} \| f\|_{L^b}\leq \|\nabla^{k} f\|_{L^b}\leq
C_3\, \lambda^{k + 2(\frac{1}{a}-\frac{1}{b})} \|f\|_{L^a},
$$
where $C_1$, $C_2$ and $C_3$ are positive constants depending on $k,\,a$ and $b$
only.
\end{lemma}

\vskip .1in
We next recall the following lower bound associated with the fractional dissipation term, which is due to Chamorro and Lemari$\rm\acute{e}$-Rieusset (see \cite[Theorem 4.2]{Chamorrolr}).

 \begin{lemma}
Let $q\in[2,\infty)$ and $\alpha\in (0,1)$, then there is a positive
constant $C(\alpha,q)$ such that
\begin{align}\label{dfoklp88}
\int_{\mathbb{R}^{2}}\Lambda^{2\alpha}f(|f|^{q-2}f)\,dx\geq C(\alpha,q)\|f\|_{\dot{B}_{q,q}^{\frac{2\alpha}{q}}}^{q}.
\end{align}
\end{lemma}
\vskip .1in
The final statement is the refined logarithmic Gronwall inequality (see \cite[Lemma 2.8]{Yeaam18}), which will be used to show the global $H^1$-bound of the solution.
\begin{lemma}\label{Lem01khj}
Assume that $l(t),\,m(t),\,n(t)$ and $f(t)$ are all nonnegative and integrable functions on $(0, T)$.
Let $A(t)$ and $B(t)$ be two
absolutely continuous and nonnegative functions on $(0, T)$ for any given $T>0$, satisfying for any $t\in (0, T)$
\begin{equation}
A'(t)+B(t)\leq \Big[l(t)+m(t)\ln(A+\alpha_{0})+n(t) {G}\big(k,\,\ln(A+B+\alpha_{0})\big)\Big](A+\alpha_{0})+f(t),\nonumber
\end{equation}
 where $\alpha_{0}=\alpha_{0}(k)\geq2$ is suitably large such that for example $\underbrace{\ln\cdot\cdot\cdot\ln}_{k\,\rm{times}}\alpha_{0}\geq 1$, and
\begin{eqnarray}  {G} (k,\,r )= r \ln r  \times\cdot\cdot\cdot\times\underbrace{\ln\cdot\cdot\cdot\ln}_{k\,\rm{times}}r,\quad k\geq1,\ k\in \mathbb{N}_{+}.\nonumber
\end{eqnarray}
Assume further that there are three constants $K\in[0,\,\infty)$, $\alpha\in[0,\,\infty)$ and $\beta\in [0,\,1)$ such that for any $t\in (0, T)$
\begin{equation}
n(t)\leq K\big(1+A(t)\big)^{\alpha}\big(1+A(t)+B(t)\big)^{\beta}.\nonumber
\end{equation}
Then the following estimate holds true for any $t\in (0, T)$
\begin{equation}
A(t)+\int_{0}^{t}{B(s)\,ds}\leq C.\nonumber
\end{equation}
\end{lemma}

\vskip .2in
\section{ The proof of Theorem \ref{Th1}}\setcounter{equation}{0}\label{sectt3}
In this section, we give the proof of Theorem \ref{Th1}.
By the classical hyperbolic method, there exists a finite time $T_{0}$ such that the system (\ref{GMHD}) is local well-posedness in the interval $[0,\,T_{0}]$ in $H^{s}$ with $s\geq2$.
Thus, it is sufficient to establish {\it a priori} estimates in the interval $[0,\,T]$ for the given $T>T_{0}$.

\vskip .1in

Firstly, we state the following basic $L^2$-energy estimate and the proof of which is straightforward.
\begin{lemma}
Let $(u_{0},b_{0})$ satisfy the conditions stated in Theorem \ref{Th1}, then it holds
\begin{eqnarray}\label{t301}
\|u(t)\|_{L^{2}}^{2}+\|b(t)\|_{L^{2}}^{2}+ \int_{0}^{t}{
(\|\Lambda^{\alpha}u(\tau)\|_{L^{2}}^{2}+\|\mathcal{L}^{\frac{1}{2}}b(\tau)\|_{L^{2}}^{2})
\,d\tau}\leq\|u_{0}\|_{L^{2}}^{2}
+\|b_{0}\|_{L^{2}}^{2},
\end{eqnarray}
where $\mathcal{L}^{\frac{1}{2}}$ is defined by
$$\widehat{\mathcal{L}^{\frac{1}{2}}b}(\xi)=\frac{|\xi|}{\sqrt{g(\xi)}} \,\widehat{b}(\xi).$$
\end{lemma}

\vskip .1in
In order to get the $H^{1}$-bound of $(u, b)$, we apply
$\nabla^{\bot}\cdot$ to the MHD equations (\ref{GMHD}) to obtain the
governing equations for the vorticity $\omega\triangleq\nabla^{\bot}\cdot u=\partial_{x_{1}}u_{2}-\partial_{x_{2}}u_{1}$ and the current $j\triangleq\nabla^{\bot}\cdot b=\partial_{x_{1}}b_{2}-\partial_{x_{2}}b_{1}$ as
follows
\begin{equation}\label{VMHD}
\left\{\aligned
&\partial_{t}\omega+(u\cdot\nabla)\omega+\Lambda^{2\alpha}\omega=(b\cdot\nabla) j,\\
&\partial_{t}j+(u\cdot\nabla)j+\mathcal{L}j=(b\cdot\nabla)\omega+T(\nabla u, \nabla b),
\endaligned \right.
\end{equation}
where
$$T(\nabla u, \nabla
b)=2\partial_{x_{1}}b_{1}(\partial_{x_{2}}u_{1}+\partial_{x_{1}}u_{2})
-2\partial_{x_{1}}u_{1}(\partial_{x_{2}}b_{1}+\partial_{x_{1}}b_{2}).$$

\vskip .1in
We now prove that the solution of (\ref{GMHD}) admits a global
$H^1$-bound as stated in the following lemma.
\begin{lemma}\label{tL32}
Let $(u_{0},b_{0})$ satisfy the conditions stated in Theorem \ref{Th1}, then it holds for any $t\in[0,\,T]$
\begin{eqnarray}\label{adt303}
\|\omega(t)\|_{L^{2}}^{2}+\|j(t)\|_{L^{2}}^{2}+ \int_{0}^{t}{(\|\Lambda^{\alpha}\omega(\tau)\|_{L^{2}}^{2}+
\|\mathcal{L}^{\frac{1}{2}}j(\tau)\|_{L^{2}}^{2})\,d\tau}\leq C(t,u_{0},b_{0}).
\end{eqnarray}
\end{lemma}

\begin{rem}\rm
After checking the proof of Lemma \ref{tL32} below, one may find that \eqref{adt303} is still true for the case $\alpha=0$.
\end{rem}

\begin{proof}
Taking the inner products of $(\ref{VMHD})_{1}$ with $\omega$, $(\ref{VMHD})_{2}$ with $j$, adding them up and using the incompressible condition, we find
\begin{align}
\frac{1}{2}\frac{d}{dt}(\|\omega(t)\|_{L^{2}}^{2}+\|j(t)\|_{L^{2}}^{2})
+\|\Lambda^{\alpha}\omega\|_{L^{2}}^{2}+\|\mathcal{L}^{\frac{1}{2}}j\|_{L^{2}}^{2}
&=\int_{\mathbb{R}^{2}} {T(\nabla u, \nabla
b)j
\,dx},\nonumber
\end{align}
where we have used the following fact
$$\int_{\mathbb{R}^{2}} {(b \cdot \nabla j) \omega
\,dx}+\int_{\mathbb{R}^{2}}  { (b \cdot \nabla\omega) j\,dx}=0.$$
By the simple interpolation inequality, it yields
\begin{align}
\int_{\mathbb{R}^{2}} {T(\nabla u, \nabla
b)j
\,dx}&\leq C\|\nabla u\|_{L^{2}}\|\nabla b\|_{L^{4}}\|j\|_{L^{4}}\nonumber\\
&\leq C\|\omega\|_{L^{2}}\|j\|_{L^{4}}^{2},\nonumber
\end{align}
which leads to
\begin{align}
\frac{1}{2}\frac{d}{dt}(\|\omega(t)\|_{L^{2}}^{2}+\|j(t)\|_{L^{2}}^{2})
+\|\Lambda^{\alpha}\omega\|_{L^{2}}^{2}+\|\mathcal{L}^{\frac{1}{2}}j\|_{L^{2}}^{2}
\leq C\|\omega\|_{L^{2}}\|j\|_{L^{4}}^{2}.\nonumber
\end{align}
In view of the Littlewood-Paley decomposition and the Bernstein inequality (see Lemma \ref{vfgty8xc}), we obtain
\begin{align}
\|j\|_{L^{4}}^{2}&\leq \|j\|_{B_{4,2}^{0}}^{2}\nonumber\\
&\leq \|\Delta_{-1}j\|_{L^{4}}^{2}+\sum_{k\geq0}\|\Delta_{k}j\|_{L^{4}}^{2}
\nonumber\\
&\leq C\|b\|_{L^{2}}^{2}+\sum_{0\leq k\leq M}\|\Delta_{k}j\|_{L^{4}}^{2}+\sum_{k> M}\|\Delta_{k}j\|_{L^{4}}^{2}
\nonumber\\
&\leq C\|b\|_{L^{2}}^{2}+\sum_{0\leq k\leq M}2^{k}\|\Delta_{k}j\|_{L^{2}}^{2}+ \sum_{k> M}2^{-k}\|\Delta_{k}\Lambda j\|_{L^{2}}^{2}
\nonumber\\
&= C\|b\|_{L^{2}}^{2}+\sum_{0\leq k\leq M}2^{k}\left\|\Delta_{k}\left(\sqrt{g(\Lambda)}\frac{j}{\sqrt{g(\Lambda)}}
\right)\right\|_{L^{2}}^{2}\nonumber\\
&\quad+ \sum_{k> M}2^{-k}\left\|\Delta_{k}\left(\sqrt{g(\Lambda)}\frac{\Lambda }{\sqrt{g(\Lambda)}}j
\right)\right\|_{L^{2}}^{2}
\nonumber\\
&\leq C\|b\|_{L^{2}}^{2}+\sum_{0\leq k\leq M}2^{k} g(2^{k}) \left\|\Delta_{k}\left(\frac{\nabla b}{\sqrt{g(\Lambda)}}
\right)\right\|_{L^{2}}^{2}\nonumber\\
&\quad+ \sum_{k> M}2^{-k} g(2^{k}) \left\|\Delta_{k}\left(\frac{\Lambda }{\sqrt{g(\Lambda)}}j
\right)\right\|_{L^{2}}^{2}
\nonumber\\
&\leq C\|b\|_{L^{2}}^{2}+C
\sum_{0\leq k\leq M}2^{k} g(2^{k})
\|\mathcal{L}^{\frac{1}{2}}b\|_{L^{2}}^{2}+ C\sum_{k> M}2^{-k} g(2^{k})
\|\mathcal{L}^{\frac{1}{2}}j\|_{L^{2}}^{2}
\nonumber\\
&\leq C\|b\|_{L^{2}}^{2}+C
\sum_{0\leq k\leq M}2^{k} g(2^{k})
\|\mathcal{L}^{\frac{1}{2}}b\|_{L^{2}}^{2}+ C\sum_{k> M}2^{-(1-\epsilon)k} \frac{g(2^{k})}{2^{k\epsilon}}
\|\mathcal{L}^{\frac{1}{2}}j\|_{L^{2}}^{2}
\nonumber\\
&\leq C\|b\|_{L^{2}}^{2}+C
\sum_{0\leq k\leq M}2^{k} g(2^{k})
\|\mathcal{L}^{\frac{1}{2}}b\|_{L^{2}}^{2}+ C\sum_{k> M}2^{-(1-\epsilon)k} \frac{g(2^{M})}{2^{M\epsilon}}
\|\mathcal{L}^{\frac{1}{2}}j\|_{L^{2}}^{2}\nonumber\\
&\leq C\|b\|_{L^{2}}^{2}+C
2^{M} g(2^{M})
\|\mathcal{L}^{\frac{1}{2}}b\|_{L^{2}}^{2}+ C2^{-(1-\epsilon)M} \frac{g(2^{M})}{2^{M\epsilon}}
\|\mathcal{L}^{\frac{1}{2}}j\|_{L^{2}}^{2}
\nonumber\\
&\leq C+C
2^{M} g(2^{M})
\|\mathcal{L}^{\frac{1}{2}}b\|_{L^{2}}^{2}+ C2^{-M}g(2^{M})
\|\mathcal{L}^{\frac{1}{2}}j\|_{L^{2}}^{2},\nonumber
\end{align}
where we have used the property \eqref{chkhgtt1}. As a result, one gets
\begin{eqnarray}
\|j\|_{L^{4}}^{2}\leq C+C
2^{M} g(2^{M})
\|\mathcal{L}^{\frac{1}{2}}b\|_{L^{2}}^{2}+ C2^{-M}g(2^{M})
\|\mathcal{L}^{\frac{1}{2}}j\|_{L^{2}}^{2}.\nonumber
\end{eqnarray}
Therefore, we get
\begin{align}\label{fghkhk4}
&\frac{1}{2}\frac{d}{dt}(\|\omega(t)\|_{L^{2}}^{2}+\|j(t)\|_{L^{2}}^{2})
+\|\Lambda^{\alpha}\omega\|_{L^{2}}^{2}+\|\mathcal{L}^{\frac{1}{2}}j\|_{L^{2}}^{2}
\nonumber\\
&\leq C\|\omega\|_{L^{2}}\left(1+
2^{M} g(2^{M})
\|\mathcal{L}^{\frac{1}{2}}b\|_{L^{2}}^{2}+  2^{-M}g(2^{M})
\|\mathcal{L}^{\frac{1}{2}}j\|_{L^{2}}^{2}\right)
\nonumber\\
&\leq C\|\omega\|_{L^{2}}\left(1+
2^{M} g(2^{M})
(1+\|\mathcal{L}^{\frac{1}{2}}b\|_{L^{2}})^{2}+  2^{-M}g(2^{M})
\|\mathcal{L}^{\frac{1}{2}}j\|_{L^{2}}^{2}\right).
\end{align}
We take $M$ large enough such that
$$
2^{M} g(2^{M})
(1+\|\mathcal{L}^{\frac{1}{2}}b\|_{L^{2}})^{2}\approx2^{-M}g(2^{M})
\|\mathcal{L}^{\frac{1}{2}}j\|_{L^{2}}^{2},$$
which implies
$$2^{M}\approx\frac{\|\mathcal{L}^{\frac{1}{2}}j\|_{L^{2}}}
{1+\|\mathcal{L}^{\frac{1}{2}}b\|_{L^{2}}}.$$
It follows from \eqref{fghkhk4} that
\begin{align}
&\frac{1}{2}\frac{d}{dt}(\|\omega(t)\|_{L^{2}}^{2}+\|j(t)\|_{L^{2}}^{2})
+\|\Lambda^{\alpha}\omega\|_{L^{2}}^{2}+\|\mathcal{L}^{\frac{1}{2}}j\|_{L^{2}}^{2}
\nonumber\\
&\leq  C\|\omega\|_{L^{2}}+C\|\omega\|_{L^{2}}
\|\mathcal{L}^{\frac{1}{2}}j\|_{L^{2}} g\left(\frac{\|\mathcal{L}^{\frac{1}{2}}j\|_{L^{2}}}
{1+\|\mathcal{L}^{\frac{1}{2}}b\|_{L^{2}}}\right)
(1+\|\mathcal{L}^{\frac{1}{2}}b\|_{L^{2}})
\nonumber\\
&\leq \frac{1}{2}
\|\mathcal{L}^{\frac{1}{2}}j\|_{L^{2}}^{2}+ C\|\omega\|_{L^{2}}+C\|\omega\|_{L^{2}}^{2} g^{2}\left(\frac{\|\mathcal{L}^{\frac{1}{2}}j\|_{L^{2}}}
{1+\|\mathcal{L}^{\frac{1}{2}}b\|_{L^{2}}}\right)
(1+\|\mathcal{L}^{\frac{1}{2}}b\|_{L^{2}})^{2},\nonumber
\end{align}
which leads to
\begin{align}\label{vbklhgw8}
& \frac{d}{dt}(\|\omega(t)\|_{L^{2}}^{2}+\|j(t)\|_{L^{2}}^{2})
+\|\Lambda^{\alpha}\omega\|_{L^{2}}^{2}+\|\mathcal{L}^{\frac{1}{2}}j\|_{L^{2}}^{2}
\nonumber\\
&\leq  C\|\omega\|_{L^{2}}+ C\|\omega\|_{L^{2}}^{2} g^{2}\left(\frac{\|\mathcal{L}^{\frac{1}{2}}j\|_{L^{2}}}
{1+\|\mathcal{L}^{\frac{1}{2}}b\|_{L^{2}}}\right)
(1+\|\mathcal{L}^{\frac{1}{2}}b\|_{L^{2}})^{2}.
\end{align}
For simplicity, we denote
$$G(\xi)\triangleq\ln(\sigma+|\xi|)\ln\ln(\sigma+|\xi|)\cdot\cdot\cdot
\underbrace{\ln\ln\cdot\cdot\cdot\ln}_{k\,\rm{times}}(\sigma+|\xi|).$$
We thus deduce from \eqref{vbklhgw8} and \eqref{logcd1} that
\begin{align}\label{vbklhgw9}
& \frac{d}{dt}(\|\omega(t)\|_{L^{2}}^{2}+\|j(t)\|_{L^{2}}^{2})
+\|\Lambda^{\alpha}\omega\|_{L^{2}}^{2}+\|\mathcal{L}^{\frac{1}{2}}j\|_{L^{2}}^{2}
\nonumber\\
&\leq   C\|\omega\|_{L^{2}}+C\|\omega\|_{L^{2}}^{2} G\left(\frac{\|\mathcal{L}^{\frac{1}{2}}j\|_{L^{2}}}
{1+\|\mathcal{L}^{\frac{1}{2}}b\|_{L^{2}}}\right)
(1+\|\mathcal{L}^{\frac{1}{2}}b\|_{L^{2}})^{2}
\nonumber\\
&\leq  C\|\omega\|_{L^{2}}+ C\|\omega\|_{L^{2}}^{2} G\left( \|\mathcal{L}^{\frac{1}{2}}j\|_{L^{2}} \right)
(1+\|\mathcal{L}^{\frac{1}{2}}b\|_{L^{2}})^{2}
\nonumber\\
&\leq C\|\omega\|_{L^{2}}+ C (1+\|\mathcal{L}^{\frac{1}{2}}b\|_{L^{2}})^{2}\|\omega\|_{L^{2}}^{2} \nonumber\\
&\quad \times \ln(\sigma+\|\mathcal{L}^{\frac{1}{2}}j\|_{L^{2}})
\ln\ln(\sigma+\|\mathcal{L}^{\frac{1}{2}}j\|_{L^{2}})\cdot\cdot\cdot
\underbrace{\ln\ln\cdot\cdot\cdot\ln}_{k\,\rm{times}}
(\sigma+\|\mathcal{L}^{\frac{1}{2}}j\|_{L^{2}}).
\end{align}
We denote
$$A(t)\triangleq \sigma+\|\omega(t)\|_{L^{2}}^{2}+\|j(t)\|_{L^{2}}^{2},\quad
B(t)\triangleq \sigma+\|\Lambda^{\alpha}\omega(t)\|_{L^{2}}^{2}
+\|\mathcal{L}^{\frac{1}{2}}j(t)\|_{L^{2}}^{2},$$
$$n(t)\triangleq (1+\|\mathcal{L}^{\frac{1}{2}}b(t)\|_{L^{2}})^{2}.$$
Then, one derives from \eqref{vbklhgw9} that
\begin{align}\label{vbklhgw10}
 \frac{d}{dt}A(t)+B(t)
&\leq   n(t)A(t)\ln(\sigma+B(t))
\ln\ln(\sigma+B(t))\cdot\cdot\cdot
\underbrace{\ln\ln\cdot\cdot\cdot\ln}_{k\,\rm{times}}
(\sigma+B(t)).
\end{align}
It is easy to check that the function $n(t)$ is in $L_{loc}^{1}(\mathbb{R}_{+})$. Moreover, it obeys
\begin{align}
n(t)&= \left(1+\Big\|\widehat{\mathcal{L}^{\frac{1}{2}}b}\Big\|_{L^{2}}\right)^{2}
\nonumber\\
&= C\left(1+\left\|\frac{|\xi|}{\sqrt{g(\xi)}} \,\widehat{b}(\xi)\right\|_{L^{2}}\right)^{2}
\nonumber\\
&\leq C\left(1+\left\||\xi|\widehat{b}(\xi)\right\|_{L^{2}}\right)^{2}
\nonumber\\
&\leq C\left(1+\left\|\nabla b\right\|_{L^{2}}\right)^{2}
\nonumber\\
&\leq C\left(1+\left\|j\right\|_{L^{2}}\right)^{2}
\nonumber\\
&\leq CA(t),\nonumber
\end{align}
which shows
\begin{align}\label{vbklhgw11}
n(t) \leq CA(t).
\end{align}
Keeping in mind \eqref{vbklhgw11} and applying the refined logarithmic Gronwall inequality (see Lemma \ref{Lem01khj}) to \eqref{vbklhgw10}, it is not hard to obtain (\ref{adt303}). Therefore, we complete the proof of Lemma \ref{tL32}.
\end{proof}

\vskip .1in
Thanks to \eqref{adt303} and the $\dot{H}^{r}$-estimate of $j$, we are able to show the following estimate, whose proof can be performed by repeating the proof of \cite[Lemma 3.4]{Zhao22}.
\begin{lemma} \label{lkjmnea33}
Let $\alpha>0$, then it holds for any $r\in(0,\alpha)$
\begin{eqnarray}\label{tcmnhqq1}
\|\Lambda^{r}j(t)\|_{L^{2}}^{2}+ \int_{0}^{t}{
\|\mathcal{L}^{\frac{1}{2}}\Lambda^{r}j(\tau)\|_{L^{2}}^{2}\,d\tau}\leq C(t,u_{0},b_{0}).
\end{eqnarray}
\end{lemma}

\begin{rem}\rm
In the case when $\alpha=0$, appealing to other approach different from  \cite[Lemma 3.4]{Zhao22}, we are still able to show that (see Theorem \ref{Th2} for details)
$$\|\Lambda^{r}j(t)\|_{L^{2}}+\|\mathcal{L}b(t)\|_{L^{2}}\leq C(t,u_{0},b_{0})$$
for any $r\in(0,1)$. This indicates that as regards Lemma \ref{lkjmnea33}, the role of $\alpha>0$ is not critical.
\end{rem}

\vskip .1in
The following result is an easy consequence of (\ref{adt303}) and \eqref{tcmnhqq1}.
\begin{lemma}
Let $\alpha>0$, then it holds for any $r\in(0,\alpha)$
\begin{eqnarray}\label{t303}
\|\omega(t)\|_{L^{2}}^{2}+\|j(t)\|_{H^{r}}^{2}+\|b(t)\|_{L^{\infty}}^{2}+ \int_{0}^{t}{(\|\Lambda^{\alpha}\omega\|_{L^{2}}^{2}+
\|\mathcal{L}^{\frac{1}{2}}j\|_{H^{r}}^{2})(\tau)\,d\tau}\leq C(t,u_{0},b_{0}).
\end{eqnarray}
\end{lemma}

\vskip .1in
To go on, we will encounter even greater difficulty. In order to introduce the reader to the main difficulty, let us consider the endpoint case, namely $\alpha=0$. In this case, our natural target is to derive the $L^{\infty}$-bound of $\omega$. Regarding the vorticity equation, to achieve this goal, the significant quantity that one needs to bound in order to get global existence is the $L^{\infty}$-norm of $\nabla j$. Unfortunately, the main difficulty encountered for global existence is due to the lack of strong dissipation in the magnetic equation because the magnetic diffusion is reducing about logarithmic diffusion from standard Laplacian diffusion. Consequently, we don't see how to estimate in a suitable way the quantity $\int_{0}^{t}\|\nabla j(\tau)\|_{L^{\infty}}\,d\tau$. From a mathematical point of view, it is extremely hard to obtain the $L^{\infty}$-norm of $\nabla j$ in this case. To say the least, deriving the $L^{q}$-bound of $\omega$ with $q\in(2,\infty)$ is also not a trivial task even the dissipation in the velocity equation is $\alpha>0$. As a matter of fact, this is our next main task. More precisely, we will show the $L^{q}$-bound of $\omega$ for any $q\in(2,\infty)$ which is a crucial component of this paper and plays an important role in obtaining the global bound for $(u,\,b)$ in $H^{s}$ with $s\geq2$. It is worthwhile to emphasize that the condition $\alpha>0$ is crucial to guarantee the validity of \eqref{t311} below.
\begin{lemma}\label{L35}
Let $(u_{0},b_{0})$ satisfy the conditions stated in Theorem \ref{Th1}. If $\alpha>0$, then it holds for any $t\in[0,\,T]$ and for any $q\in (2,\infty)$
\begin{eqnarray}\label{t311}
\|\omega(t)\|_{L^{q}}+\int_{0}^{t}{
\|\nabla j(\tau)\|_{L^{q}} \,d\tau}\leq C(t,u_{0},b_{0}).
\end{eqnarray}
In particular, it holds
\begin{eqnarray}\label{tffgklq1}
\int_{0}^{t}{
\|\nabla b(\tau)\|_{L^{\infty}} \,d\tau}\leq C(t,u_{0},b_{0}).
\end{eqnarray}
\end{lemma}
\begin{proof}
It follows from $\eqref{VMHD}_{2}$ that
$$\partial_{t}(\Lambda j)+\mathcal{L}\Lambda j=\Lambda\partial_{x_{i}}(b_{i}\omega-u_{i}j)+\Lambda T(\nabla u, \nabla b),$$
which along with its Fourier transform reads
$$\partial_{t}\widehat{(\Lambda j)}(\xi,t)+\frac{|\xi|^{2}}{g(\xi)}\widehat{(\Lambda j)}(\xi,t)=\mathcal{F}(
\Lambda\partial_{x_{i}}(b_{i}\omega-u_{i}j))(\xi,t)+\mathcal{F}(\Lambda T(\nabla u, \nabla b))(\xi,t).$$
As a result, we have
\begin{align}\label{tffmpwa8}
\widehat{(\Lambda j)}(\xi,t)=&e^{\frac{-t|\xi|^{2}}{g(\xi)}}\widehat{(\Lambda j_{0})}(\xi)+\int_{0}^{t}e^{\frac{-(t-\tau)|\xi|^{2}}{g(\xi)}}\mathcal{F}(
\Lambda\partial_{x_{i}}(b_{i}\omega-u_{i}j))(\xi,\tau)\,d\tau\nonumber\\&+
\int_{0}^{t}e^{\frac{-(t-\tau)|\xi|^{2}}{g(\xi)}}\mathcal{F}(\Lambda T(\nabla u, \nabla b))(\xi,\tau)\,d\tau\nonumber\\
=&e^{\frac{-t|\xi|^{2}}{g(\xi)}}|\xi|\widehat{j_{0}}(\xi)+\int_{0}^{t}
|\xi|^{2-\delta}e^{\frac{-(t-\tau)|\xi|^{2}}{g(\xi)}}|\xi|^{\delta-2}\mathcal{F}(
\Lambda\partial_{x_{i}}(b_{i}\omega-u_{i}j))(\xi,\tau)\,d\tau\nonumber\\&+
\int_{0}^{t}|\xi|^{1-\delta}e^{\frac{-(t-\tau)|\xi|^{2}}{g(\xi)}}|\xi|^{\delta-1}
\mathcal{F}(\Lambda T(\nabla u, \nabla b))(\xi,\tau)\,d\tau
\nonumber\\
=&|\xi|^{2-\delta}e^{\frac{-t|\xi|^{2}}{g(\xi)}}\widehat{\Lambda^{\delta-1}j_{0}}(\xi)
+\int_{0}^{t}
|\xi|^{2-\delta}e^{\frac{-(t-\tau)|\xi|^{2}}{g(\xi)}}\mathcal{F}(
\Lambda^{\delta-1}\partial_{x_{i}}(b_{i}\omega-u_{i}j))(\xi,\tau)\,d\tau\nonumber\\&+
\int_{0}^{t}|\xi|^{1-\delta}e^{\frac{-(t-\tau)|\xi|^{2}}{g(\xi)}}
\mathcal{F}(\Lambda^{\delta} T(\nabla u, \nabla b))(\xi,\tau)\,d\tau,
\end{align}
where the constant $\delta>0$ will be fixed later.
We thus derive from \eqref{tffmpwa8} that
\begin{align}
\widehat{(\Lambda j)}(\xi,t) =&|\xi|^{2-\delta}e^{\frac{-t|\xi|^{2}}{g(\xi)}}\widehat{\Lambda^{\delta-1}j_{0}}(\xi)
+\int_{0}^{t}
|\xi|^{2-\delta}e^{\frac{-(t-\tau)|\xi|^{2}}{g(\xi)}}\mathcal{F}(
\Lambda^{\delta-1}\partial_{x_{i}}(b_{i}\omega-u_{i}j))(\xi,\tau)\,d\tau\nonumber\\&+
\int_{0}^{t}|\xi|^{1-\delta}e^{\frac{-(t-\tau)|\xi|^{2}}{g(\xi)}}
\mathcal{F}(\Lambda^{\delta} T(\nabla u, \nabla b))(\xi,\tau)\,d\tau.\nonumber
\end{align}
Applying inverse Fourier transform to above equality shows
\begin{align}\label{tmkgfd68}
\Lambda j(t) =&\Lambda^{2-\delta} K(t)\ast \Lambda^{\delta-1}j_{0}
+\int_{0}^{t}\Lambda^{2-\delta} K(t-\tau)\ast(
\Lambda^{\delta-1}\partial_{x_{i}}(b_{i}\omega-u_{i}j))(\tau)\,d\tau\nonumber\\&+
\int_{0}^{t}\Lambda^{1-\delta} K(t-\tau)\ast
(\Lambda^{\delta} T(\nabla u, \nabla b))(\tau)\,d\tau.
\end{align}
Taking $L^{q}$-norm (for any $q\in(2,\infty)$) to the both sides of the above equality in terms of space variable and using the Young inequality, it follows that
\begin{align}
\|\nabla j(t)\|_{L^{q}}&\leq C\|\Lambda j(t)\|_{L^{q}}\nonumber\\&\leq
 C\|\Lambda^{2-\delta} K(t)\ast \Lambda^{\delta-1}j_{0}\|_{L^{q}}\nonumber\\&\quad
 +C\int_{0}^{t}{\|\Lambda^{2-\delta} K(t-\tau)\ast \{ \Lambda^{\delta-1}\partial_{x_{i}}(b_{i}\omega-u_{i}j))\}(\tau)\|_{L^{q}}\,d\tau}
 \nonumber\\&\quad+C
\int_{0}^{t}{\|\Lambda^{1-\delta} K(t-\tau)\ast  \Lambda^{\delta}T(\nabla u, \nabla b)(\tau)\|_{L^{q}}\,d\tau}
\nonumber\\&\leq C\|\Lambda^{2-{\delta}}K(t)\|_{L^{1}}\| \Lambda^{\delta-1}j_{0}\|_{L^{q}}
\nonumber\\&\quad+C\int_{0}^{t}{\|\Lambda^{2-\delta} K(t-\tau)\|_{L^{1}}(\|\Lambda^{\delta}(b \omega)(\tau)\|_{L^{q}}+\|\Lambda^{\delta}
(uj)(\tau)\|_{L^{q}})\,d\tau}
\nonumber\\&\quad
+C\int_{0}^{t}{\|\Lambda^{1-\delta}K(t-\tau)\|_{L^{2}} \|\Lambda^{\delta}T(\nabla u, \nabla b)(\tau)\|_{L^{\frac{2q}{q+2}}}\,d\tau}
\nonumber\\&\leq C\|\Lambda^{2-\delta}K(t)\|_{L^{1}}\| \Lambda^{\delta}b_{0}\|_{L^{q}}
\nonumber\\&\quad+C\int_{0}^{t}{\|\Lambda^{2-\delta} K(t-\tau)\|_{L^{1}}(\|\Lambda^{\delta}(b\otimes \omega)(\tau)\|_{L^{q}}+\|\Lambda^{\delta}
(u\otimes j)(\tau)\|_{L^{q}})\,d\tau}
\nonumber\\&\quad
+C\int_{0}^{t}{\|\Lambda^{1-\delta}K(t-\tau)\|_{L^{2}} \|\Lambda^{\delta}T(\nabla u, \nabla b)(\tau)\|_{L^{\frac{2q}{q+2}}}\,d\tau}.\nonumber
\end{align}
Taking $L^{1}$-norm in terms of time variable and appealing to the convolution Young inequality, one concludes by using \eqref{bnmper8} that
\begin{align}\label{t312}
\|\nabla j(t)\|_{L_{t}^{1}L^{q}} \leq&C\|\Lambda^{2-\delta} K(t)\|_{L_{t}^{1}L^{1}} +C\|\Lambda^{1-\delta} K(t)\|_{L_{t}^{1}L^{2}}\|\Lambda^{\delta}T(\nabla u, \nabla b)\|_{L_{t}^{1}L^{\frac{2q}{q+2}}}\nonumber\\ &+C\|\Lambda^{2-\delta} K(t)\|_{L_{t}^{1}L^{1}}(\|\Lambda^{\delta}(b\omega)\|_{L_{t}^{1}L^{q}}
+\|\Lambda^{\delta}(uj)\|_{L_{t}^{1}L^{q}})\nonumber\\ \leq&C +C\|\Lambda^{\delta}T(\nabla u, \nabla b)\|_{L_{t}^{1}L^{\frac{2q}{q+2}}}+C(\|\Lambda^{\delta}(b\omega)\|_{L_{t}^{1}L^{q}}
+\|\Lambda^{\delta}(uj)\|_{L_{t}^{1}L^{q}}).
\end{align}
Taking $\delta<\alpha$ and using \eqref{t303}, we may derive
\begin{align}\label{bnhtdf01}
C\|\Lambda^{\delta}T(\nabla u, \nabla b)\|_{L_{t}^{1}L^{\frac{2q}{q+2}}}&\leq
C\|\Lambda^{\delta}\nabla u\|_{L_{t}^{2}L^{2}}\| \nabla b\|_{L_{t}^{2}L^{q}}
+C\|\Lambda^{\delta}\nabla b\|_{L_{t}^{\infty}L^{2}}\|\nabla u\|_{L_{t}^{1}L^{q}}
\nonumber\\
&\leq
C\|\Lambda^{\delta}\omega\|_{L_{t}^{2}L^{2}}\|j\|_{L_{t}^{2}L^{q}}
+C\|\Lambda^{\delta}j\|_{L_{t}^{\infty}L^{2}}\|\omega\|_{L_{t}^{1}L^{q}}
\nonumber\\
&\leq
C+C\|\omega\|_{L_{t}^{1}L^{q}}.
\end{align}
Similarly, one has by taking $\delta<\min\{\alpha,\frac{1}{q}\}$ and using \eqref{t303} that
\begin{align}\label{bnhtdf02}
C(\|\Lambda^{\delta}(b\omega)\|_{L_{t}^{1}L^{q}}
+\|\Lambda^{\delta}(uj)\|_{L_{t}^{1}L^{q}})\leq&
C(\|b\|_{L_{t}^{\infty}L^{\infty}}\|\Lambda^{\delta}\omega\|_{L_{t}^{1}L^{q}}+
\|\Lambda^{\delta}b\|_{L_{t}^{\infty}L^{\frac{2}{\delta}}}
\|\omega\|_{L_{t}^{1}L^{\frac{2q}{2-\delta q}}}
\nonumber\\&+\|\Lambda^{\delta}u\|_{L_{t}^{2}L^{2q}}
\|j\|_{L_{t}^{2}L^{2q}}+\|u\|_{L_{t}^{2}L^{2q}}
\|\Lambda^{\delta}j\|_{L_{t}^{2}L^{2q}})\nonumber\\
\leq&
C(\|b\|_{L_{t}^{\infty}L^{\infty}}\|\Lambda^{\delta}\omega\|_{L_{t}^{1}L^{q}}+
\|j\|_{L_{t}^{\infty}L^{2}}
\|\Lambda^{\delta}\omega\|_{L_{t}^{1}L^{q}})
\nonumber\\&+C(\|u\|_{L_{t}^{2}L^{2}}+\|\omega\|_{L_{t}^{2}L^{2}})
(\|j\|_{L_{t}^{2}L^{2}}+\|\mathcal{L}^{\frac{1}{2}}j\|_{L_{t}^{2}L^{2}})\nonumber\\
\leq&
C+C\|\Lambda^{\delta}\omega\|_{L_{t}^{1}L^{q}}.
\end{align}
Putting \eqref{bnhtdf01}  and \eqref{bnhtdf02} into \eqref{t312} shows
\begin{align}\label{bnhtdf03}
\|\nabla j(t)\|_{L_{t}^{1}L^{q}} \leq&C+C\|\omega\|_{L_{t}^{1}L^{q}}+
C\|\Lambda^{\delta}\omega\|_{L_{t}^{1}L^{q}}.
\end{align}
To bound the terms at the righthand side of \eqref{bnhtdf03}, we deduce from
$\eqref{VMHD}_{1}$ that
\begin{align}\label{bnhtdf04}
\frac{1}{q}\frac{d}{dt}\|\omega(t)\|_{L^{q}}^{q} +\int_{\mathbb{R}^{2}}\Lambda^{2\alpha}\omega(|\omega|^{q-2}\omega)\,dx \leq&C\|b\|_{L^{\infty}}\|\nabla j\|_{L^{q}}\|\omega\|_{L^{q}}^{q-1} \nonumber\\
\leq&C \|\nabla j\|_{L^{q}}\|\omega\|_{L^{q}}^{q-1}.
\end{align}
Thanks to \eqref{dfoklp88}, one obtains
\begin{align}\label{bnhtdf05}
\int_{\mathbb{R}^{2}}\Lambda^{2\alpha}\omega(|\omega|^{q-2}\omega)\,dx\geq C(\alpha,q)\|\omega\|_{\dot{B}_{q,q}^{\frac{2\alpha}{q}}}^{q}.
\end{align}
Putting \eqref{bnhtdf05} into \eqref{bnhtdf04} leads to
\begin{align}\label{bnhtdf06}
\frac{1}{q}\frac{d}{dt}\|\omega(t)\|_{L^{q}}^{q} +C\|\omega\|_{\dot{B}_{q,q}^{\frac{2\alpha}{q}}}^{q} \leq&C \|\nabla j\|_{L^{q}}\|\omega\|_{L^{q}}^{q-1}.
\end{align}
Obviously, it follows from \eqref{bnhtdf06} that
\begin{align}
 \frac{d}{dt}\|\omega(t)\|_{L^{q}}  \leq&C \|\nabla j\|_{L^{q}},\nonumber
\end{align}
which implies
\begin{align}\label{bnhtdf07}
\|\omega(t)\|_{L^{q}}  \leq& \|\omega_{0}\|_{L^{q}} +C\|\nabla j\|_{L_{t}^{1}L^{q}}.
\end{align}
Keeping in mind \eqref{bnhtdf06} and making use of \eqref{bnhtdf07}, we are able to show
\begin{align}
\int_{0}^{t}\|\omega(\tau)\|_{\dot{B}_{q,q}^{\frac{2\alpha}{q}}}^{q}\,d\tau
&\leq C+C\int_{0}^{t}\|\nabla j(\tau)\|_{L^{q}}\|\omega(\tau)\|_{L^{q}}^{q-1}\,d\tau\nonumber\\
&\leq C+C\int_{0}^{t}\|\nabla j(\tau)\|_{L^{q}}\left(1+\|\nabla j\|_{L_{\tau}^{1}L^{q}}\right)^{q-1}\,d\tau\nonumber\\
&\leq C+C\left(1+\|\nabla j\|_{L_{t}^{1}L^{q}}\right)^{q-1}\int_{0}^{t}\|\nabla j(\tau)\|_{L^{q}}\,d\tau
\nonumber\\
&\leq C+C\|\nabla j\|_{L_{t}^{1}L^{q}}^{q},\nonumber
\end{align}
which gives
\begin{align}\label{bnhtdf08}
\|\omega(\tau)\|_{L_{t}^{q}\dot{B}_{q,q}^{\frac{2\alpha}{q}}} \leq C+C\|\nabla j\|_{L_{t}^{1}L^{q}}.
\end{align}
Coming back to \eqref{bnhtdf03} and using \eqref{bnhtdf08}, we derive by further restricting $\delta<\frac{2\alpha}{q}$ that
\begin{align}
\|\nabla j(t)\|_{L_{t}^{1}L^{q}} \leq&C+C\|\omega\|_{L_{t}^{1}L^{q}}+
C\|\Lambda^{\delta}\omega\|_{L_{t}^{1}L^{q}}\nonumber\\
\leq&C+C\|\omega\|_{L_{t}^{1}L^{q}}+
C\int_{0}^{t}\|\omega(\tau)\|_{L^{2}}^{1-\frac{(1+\delta)q-2}{q-2(1-\alpha)}}
\|\omega(\tau)\|_{\dot{B}_{q,q}^{\frac{2\alpha}{q}}}
^{\frac{(1+\delta)q-2}{q-2(1-\alpha)}}\,d\tau
\nonumber\\
\leq&C+C\|\omega\|_{L_{t}^{1}L^{q}}+
C\left(\int_{0}^{t}\|\omega(\tau)\|_{L^{2}}^{\frac{(2\alpha-\delta q)q}{q[q-2(1-\alpha)]-(1+\delta)q+2}}
\,d\tau\right)^{\frac{q[q-2(1-\alpha)]-(1+\delta)q+2}{q[q-2(1-\alpha)]}}
\nonumber\\&\times\left(\int_{0}^{t}\|\omega(\tau)\|_{\dot{B}_{q,q}^{\frac{2\alpha}{q}}}
^{q}\,d\tau\right)^{\frac{(1+\delta)q-2}{q[q-2(1-\alpha)]}}
\nonumber\\
\leq&C+C\|\omega\|_{L_{t}^{1}L^{q}}+
C\|\omega(\tau)\|_{L_{t}^{q}\dot{B}_{q,q}^{\frac{2\alpha}{q}}}
^{\frac{(1+\delta)q-2}{q-2(1-\alpha)}}
\nonumber\\
\leq&C+C\|\omega\|_{L_{t}^{1}L^{q}}+
C\|\nabla j\|_{L_{t}^{1}L^{q}}
^{\frac{(1+\delta)q-2}{q-2(1-\alpha)}}
\nonumber\\
\leq&
\frac{1}{2}\|\nabla j(t)\|_{L_{t}^{1}L^{q}}+C+C\|\omega\|_{L_{t}^{1}L^{q}},\nonumber
\end{align}
which yields
\begin{align}\label{bnhtdf09}
\|\nabla j(t)\|_{L_{t}^{1}L^{q}} \leq&C+C\|\omega\|_{L_{t}^{1}L^{q}}.
\end{align}
Denoting
$$A(t)\triangleq \|\nabla j(t)\|_{L_{t}^{1}L^{q}},$$
we deduce from \eqref{bnhtdf07} and \eqref{bnhtdf09} that
$$A(t)\leq C+C\int_{0}^{t}A(\tau)\,d\tau.$$
Applying a Gronwall-type argument to the above inequality, we obtain
$$A(t)\leq C(t,u_{0},b_{0}),$$
which further implies for any $q\in (2,\infty)$ that
$$\|\omega(t)\|_{L^{q}}+\int_{0}^{t}\|\nabla j(\tau)\|_{L^{q}}\,d\tau\leq C(t,u_{0},b_{0}).$$
This achieves the proof of Lemma \ref{L35}.
\end{proof}

\vskip .1in
Finally, we are in a position to prove Theorem \ref{Th1}.
\begin{proof}[{Proof of Theorem \ref{Th1}}]
With a priori estimates achieved above, it is a standard procedure to complete
the proof of Theorem \ref{Th1}. The first step is to construct a local-in-time solution which can be achieved by quite standard arguments and is thus
omitted here. The last step is to show the global $H^{s}$-bound.
To this end, we apply $\Lambda^{s}$ to the equations $u$ and $b$, and take the $L^2$ inner product of the resulting equations with $(\Lambda^{s}u,\,\Lambda^{s}b)$ to obtain
\begin{align}
& \frac{1}{2}\frac{d}{dt}(\|\Lambda^{s}u(t)\|_{L^{2}}^{2}+\|\Lambda^{s}b(t)\|_{L^{2}}^{2})
+\|\Lambda^{s+\alpha}u\|_{L^{2}}^{2}
+\|\mathcal{L}^{\frac{1}{2}}\Lambda^{s}b\|_{L^{2}}^{2} \nonumber\\
&=  -\int_{\mathbb{R}^{2}}{[\Lambda^{s}, u\cdot\nabla]u\cdot
\Lambda^{s}u\,dx}+\int_{\mathbb{R}^{2}}{[\Lambda^{s},
b\cdot\nabla]b\cdot
\Lambda^{s}u\,dx}\nonumber\\
& \quad- \int_{\mathbb{R}^{2}}{[\Lambda^{s}, u\cdot\nabla]b\cdot
\Lambda^{s}b\,dx}+\int_{\mathbb{R}^{2}}{[\Lambda^{s},
b\cdot\nabla]u\cdot
\Lambda^{s}b\,dx}\nonumber\\
&=  J_{1}+J_{2}+J_{3}+J_{4},\nonumber
\end{align}
where $[a,\,b]$ is the standard commutator notation, namely $[a,\,b]=ab-ba$.
To handle the four terms, we need the following Kato-Ponce inequality (see \cite{kaPonce})
$$\|[\Lambda^{s}, f]g\|_{L^{p}}\leq C(\|\nabla f\|_{L^{p_{1}}}\|\Lambda^{s-1}g\|_{L^{p_{2}}}+\|g\|_{L^{p_{3}}
}\|\Lambda^{s}f\|_{L^{p_{4}}}),$$
where $1<p_{1},\,p_{3}\leq\infty$ and $1<p_{2},\,p_{4}<\infty$ satisfy $\frac{1}{p}=\frac{1}{p_{1}}+\frac{1}{p_{2}}=\frac{1}{p_{3}}+\frac{1}{p_{4}}$.
As a result, one can deduce that
\begin{align}
J_{1}&\leq C\|\nabla u\|_{L^{q}}\|\Lambda^{s}u\|_{L^{\frac{2q}{q-1}}}^{2}
\nonumber\\
&\leq C\|\nabla u\|_{L^{q}}\|\Lambda^{s}u\|_{L^{2}}^{2-\frac{2}{\alpha q}}\|\Lambda^{s+\alpha}u\|_{L^{2}}^{ \frac{2}{\alpha q}}
\nonumber\\
&\leq
\frac{1}{2}\|\Lambda^{s+\alpha}u\|_{L^{2}}^{2}+C\|\omega\|_{L^{q}}^{\frac{\alpha q}{\alpha q-1}}\|\Lambda^{s}u\|_{L^{2}}^{2},\nonumber
\end{align}
$$J_{2}\leq C\|\nabla b\|_{L^{\infty}}(\|\Lambda^{s}u\|_{L^{2}}^{2}+\|\Lambda^{s}b\|_{L^{2}}^{2}),$$
\begin{align}
J_{3} &\leq C\|[\Lambda^{s}, u\cdot\nabla]b\|_{L^{2}}
\|\Lambda^{s}b\|_{L^{2}}\nonumber\\
&\leq C(\|\nabla u\|_{L^{q}}\|\Lambda^{s}b\|_{L^{\frac{2q}{q-2}}}+\|\nabla b\|_{L^{\infty}}\|\Lambda^{s}u\|_{L^{2}})
\|\Lambda^{s}b\|_{L^{2}}
\nonumber\\
&\leq C(\|\nabla u\|_{L^{q}}\|\Lambda^{s}b\|_{L^{2}}^{1-\eta}
\|\mathcal{L}^{\frac{1}{2}}b\|_{H^{s}}^{\eta}+\|\nabla b\|_{L^{\infty}}\|\Lambda^{s}u\|_{L^{2}})
\|\Lambda^{s}b\|_{L^{2}}
\nonumber\\
&\leq
\frac{1}{2}\|\mathcal{L}^{\frac{1}{2}}\Lambda^{s}b\|_{L^{2}}^{2}+
C\left(\|\nabla b\|_{L^{\infty}}+\|\omega\|_{L^{q}}^{\frac{2}{1-\eta}}\right)
(\|\Lambda^{s}u\|_{L^{2}}^{2}+\|\Lambda^{s}b\|_{L^{2}}^{2}),\nonumber
\end{align}
\begin{align}
J_{4}&\leq C\|[\Lambda^{s}, b\cdot\nabla]u\|_{L^{2}}
\|\Lambda^{s}b\|_{L^{2}}\nonumber\\
&\leq C(\|\nabla b\|_{L^{\infty}}\|\Lambda^{s}u\|_{L^{2}}+\|\nabla u\|_{L^{q}}\|\Lambda^{s}b\|_{L^{\frac{2q}{q-2}}})
\|\Lambda^{s}b\|_{L^{2}}
\nonumber\\
&\leq
\frac{1}{2}\|\mathcal{L}^{\frac{1}{2}}\Lambda^{s}b\|_{L^{2}}^{2}+
C\left(\|\nabla b\|_{L^{\infty}}+\|\omega\|_{L^{q}}^{\frac{2}{1-\eta}}\right)
(\|\Lambda^{s}u\|_{L^{2}}^{2}+\|\Lambda^{s}b\|_{L^{2}}^{2}).\nonumber
\end{align}
Consequently, we get
\begin{align}&\label{vsjhf95}
 \frac{d}{dt}(\|\Lambda^{s}u(t)\|_{L^{2}}^{2}+\|\Lambda^{s}b(t)\|_{L^{2}}^{2})
 +\|\Lambda^{s+\alpha}u\|_{L^{2}}^{2}
+\|\mathcal{L}^{\frac{1}{2}}\Lambda^{s}b\|_{L^{2}}^{2}
\nonumber\\&\leq C(\|\omega\|_{L^{q}}^{\frac{\alpha q}{\alpha q-1}}+\|\omega\|_{L^{q}}^{\frac{2}{1-\eta}}+\|\nabla b\|_{L^{\infty}})(\|\Lambda^{s}u\|_{L^{2}}^{2}+\|\Lambda^{s}b\|_{L^{2}}^{2}).
\end{align}
Applying the Gronwall type inequality to \eqref{vsjhf95} and using the estimates \eqref{t311} as well as \eqref{tffgklq1}, we eventually obtain
$$\|\Lambda^{s}u(t)\|_{L^{2}}+\|\Lambda^{s}b(t)\|_{L^{2}}+\int_{0}^{t}
(\|\Lambda^{s+\alpha}u(\tau)\|_{L^{2}}^{2}
+\|\mathcal{L}^{\frac{1}{2}}\Lambda^{s}b(\tau)\|_{L^{2}}^{2})\,d\tau\leq C(t,u_{0},b_{0}),$$
which along with \eqref{t301} yields the desired global $H^{s}$-bound. Thanks to the global $H^{s}$-bound with $s\geq2$, we thus get that both $u$ and $b$ are in $L_{loc}^{1}(\mathbb{R}_{+}; \rm {Lip(\mathbb{R}^{2})})$, which imply the uniqueness immediately, and thus omit the details. Consequently, we complete the proof of Theorem \ref{Th1}.
\end{proof}

\vskip .2in
\section{ The proof of Theorem \ref{Th2}}\setcounter{equation}{0}\label{addedxq}
This section mainly focuses on the proof of Theorem \ref{Th2}. Precisely, we are going to derive some global \emph{a prior} bounds for the case $\alpha=0$, which of course improve the one of \eqref{t303}. We believe these bounds will play important roles in the eventual solution of the global regularity problem for the 2D resistive MHD equations. More precisely, the global \emph{a prior} bounds of Theorem \ref{Th2} can be stated one by one as follows.
\begin{lemma}  \label{vsadlea11}
Let $(u_{0},b_{0})$ satisfy the conditions stated in Theorem \ref{Th2}, then it holds for any $r\in[0,1)$
\begin{eqnarray}\label{vsadfg11}
\|\Lambda^{r}j(t)\|_{L^{2}} \leq C(t,u_{0},b_{0}).
\end{eqnarray}
In particular, \eqref{vsadfg11} implies
\begin{eqnarray}\label{68vsadfgg}
\|b(t)\|_{L^{\infty}} \leq C(t,u_{0},b_{0}).
\end{eqnarray}
Moreover, there also holds
\begin{eqnarray}\label{adfkh1}
\|\mathcal{L}b(t)\|_{L^{2}} \leq C(t,u_{0},b_{0}).
\end{eqnarray}
\end{lemma}

\begin{proof}
Similar to \eqref{tmkgfd68}, we may derive
\begin{align}
\Lambda^{r} j(t) =&\Lambda^{r} K(t)\ast  j_{0}
+\int_{0}^{t}\Lambda^{r+1+\epsilon} K(t-\tau)\ast
\Lambda^{-\epsilon-1}\partial_{x_{i}}\big(b_{i}\omega-u_{i}j\big)(\tau)\,d\tau \nonumber\\&+
\int_{0}^{t}\Lambda^{r} K(t-\tau)\ast
T(\nabla u, \nabla b)(\tau)\,d\tau,\nonumber
\end{align}
which yields
\begin{align}\label{vsadfg12}
\|\Lambda^{r} j(t)\|_{L^{2}} \leq&\|\Lambda^{r} K(t)\ast  j_{0}\|_{L^{2}}
+\int_{0}^{t}\|\Lambda^{r+1+\epsilon} K(t-\tau)\ast
\Lambda^{-\epsilon-1}\partial_{x_{i}}\big(b_{i}\omega-u_{i}j\big)(\tau)\|_{L^{2}}\,d\tau \nonumber\\&+
\int_{0}^{t}\|\Lambda^{r} K(t-\tau)\ast
T(\nabla u, \nabla b)(\tau)\|_{L^{2}}\,d\tau.
\end{align}
Thanks to the Plancherel Theorem, one has
\begin{align}\label{ydvba11}
\|\Lambda^{r} K(t)\ast  j_{0}\|_{L^{2}}=& \left\|\widehat{\Lambda^{r} K(t)} \widehat{j_{0}}\right\|_{L^{2}}\nonumber\\=& \left\||\xi|^{r} e^{\frac{-t|\xi|^{2}}{g(\xi)}} \widehat{j_{0}}\right\|_{L^{2}}
\nonumber\\ \leq & \left\||\xi|^{r} \widehat{j_{0}}\right\|_{L^{2}}
\nonumber\\=& \left\|\Lambda^{r}j_{0}\right\|_{L^{2}}.
\end{align}
Due to $r<1$, we take $\epsilon\in (0,1-r)$, then it follows from \eqref{bnmper8} and the convolution inequality in space variables that
\begin{align}
&\int_{0}^{t}\|\Lambda^{r+1+\epsilon} K(t-\tau)\ast
\Lambda^{-\epsilon-1}\partial_{x_{i}}\big(b_{i}\omega-u_{i}j\big)(\tau)\|_{L^{2}}\,d\tau
\nonumber\\ &\leq  C
\int_{0}^{t}\|\Lambda^{r+1+\epsilon} K(t-\tau)\|_{L^{1}}
\|\Lambda^{-\epsilon-1}\partial_{x_{i}}\big(b_{i}\omega-u_{i}j\big)(\tau)\|_{L^{2}}\,d\tau
\nonumber\\ &\leq  C
\int_{0}^{t}\|\Lambda^{r+1+\epsilon} K(t-\tau)\|_{L^{1}}
\|\Lambda^{-\epsilon}\big(b_{i}\omega-u_{i}j\big)(\tau)\|_{L^{2}}\,d\tau
\nonumber\\ &\leq  C
\int_{0}^{t}\|\Lambda^{r+1+\epsilon} K(t-\tau)\|_{L^{1}}
(\|b\omega(\tau)\|_{L^{\frac{2}{1+\epsilon}}}+\|u j(\tau)\|_{L^{\frac{2}{1+\epsilon}}})\,d\tau
\nonumber\\ &\leq  C
\int_{0}^{t}\|\Lambda^{r+1+\epsilon} K(t-\tau)\|_{L^{1}}
(\|b\|_{L^{\frac{2}{\epsilon}}}\|\omega\|_{L^{2}}+
\|u\|_{L^{\frac{2}{\epsilon}}}\|j\|_{L^{2}})(\tau)\,d\tau
\nonumber\\ &\leq  C
\int_{0}^{t}\|\Lambda^{r+1+\epsilon} K(t-\tau)\|_{L^{1}}
(\|b\|_{L^{2}}^{\epsilon}\|j\|_{L^{2}}^{1-\epsilon}\|\omega\|_{L^{2}}+
\|u\|_{L^{2}}^{\epsilon}\|\omega\|_{L^{2}}^{1-\epsilon}
\|j\|_{L^{2}})(\tau)\,d\tau
\nonumber\\ &\leq  C(t)
\int_{0}^{t}\|\Lambda^{r+1+\epsilon} K(\tau)\|_{L^{1}}\,d\tau
\nonumber\\ &\leq  C(t),\nonumber
\end{align}
where we have used \eqref{t301} and \eqref{adt303}.
Using \eqref{bnmper8} and the convolution inequality in space variables again, we have
\begin{align}
\int_{0}^{t}\|\Lambda^{r} K(t-\tau)\ast
T(\nabla u, \nabla b)(\tau)\|_{L^{2}}\,d\tau &\leq  C\int_{0}^{t}\|\Lambda^{r} K(t-\tau)\|_{L^{2}}
\|T(\nabla u, \nabla b)(\tau)\|_{L^{1}}\,d\tau\nonumber\\&\leq  C\int_{0}^{t}\|\Lambda^{r} K(t-\tau)\|_{L^{2}}
\|\nabla u(\tau)\|_{L^{2}}\|\nabla b(\tau)\|_{L^{2}}\,d\tau
\nonumber\\&\leq  C\int_{0}^{t}\|\Lambda^{r} K(t-\tau)\|_{L^{2}}
\|\omega(\tau)\|_{L^{2}}\|j(\tau)\|_{L^{2}}\,d\tau
\nonumber\\&\leq  C(t)\int_{0}^{t}\|\Lambda^{r} K(\tau)\|_{L^{2}}\,d\tau
\nonumber\\&\leq  C(t).\nonumber
\end{align}
Inserting the above estimates into \eqref{vsadfg12} gives
\begin{align}
\|\Lambda^{r} j(t)\|_{L^{2}}\leq  C(t),\nonumber
\end{align}
which is \eqref{vsadfg11}. The proof of \eqref{adfkh1} is much involved and the argument in dealing with  \eqref{vsadfg11} does not work. To this end, we first show for any $q\in[2,\infty)$
\begin{eqnarray}\label{adyhye004}
\|\mathcal{L}b(t)\|_{L_{t}^{q}L^{2}} \leq C(t,u_{0},b_{0}).
\end{eqnarray}
The equation $\eqref{ZEGMHD}_{2}$ is rewritten as
\begin{eqnarray}\label{adyhye005}
\partial_{t}b+\mathcal{L}b=(b\cdot\nabla)u-(u\cdot\nabla)b\triangleq f.
\end{eqnarray}
Obviously, we get
\begin{align}
\|f\|_{L^{2}}&\leq
\|(b\cdot\nabla)u\|_{L^{2}}+\|(u\cdot\nabla)b\|_{L^{2}}
\nonumber\\&\leq
C(\|b\|_{L^{\infty}}\|\nabla u\|_{L^{2}}+\|u\|_{L^{4}}\|\nabla b\|_{L^{4}}),\nonumber
\end{align}
which along with \eqref{t301}, \eqref{adt303}, \eqref{vsadfg11} and \eqref{68vsadfgg} yield
\begin{eqnarray}\label{adyhye007}
\|f(t)\|_{L_{t}^{\infty}L^{2}} \leq C(t,u_{0},b_{0}).
\end{eqnarray}
Taking the $L^2$ product of both sides of \eqref{adyhye005} with $\mathcal{L}b$, we are able to show
\begin{align}
\frac{1}{2}\frac{d}{dt}\|\mathcal{L}^{\frac{1}{2}}b(t)\|_{L^{2}}^{2}
+\|\mathcal{L}b\|_{L^{2}}^{2}&=\int_{\mathbb{R}^{2}}{
f
\cdot \mathcal{L} b\,dx}\nonumber\\&\leq
\|f\|_{L^{2}}\|\mathcal{L} b\|_{L^{2}}
\nonumber\\&\leq \frac{1}{2}\|\mathcal{L}b\|_{L^{2}}^{2}+
C\|f\|_{L^{2}}^{2},\nonumber
\end{align}
which yields
$$\frac{d}{dt}\|\mathcal{L}^{\frac{1}{2}}b(t)\|_{L^{2}}^{2}
+\|\mathcal{L}b\|_{L^{2}}^{2}\leq
C\|f\|_{L^{2}}^{2}.$$
Performing a time integration and using \eqref{adyhye007}, we have
\begin{eqnarray} \label{adyhye008}
\|\mathcal{L}b(t)\|_{L_{t}^{2}L^{2}} \leq C(t,u_{0},b_{0}).
\end{eqnarray}
It follows from \eqref{adyhye005}, \eqref{adyhye007} and \eqref{adyhye008} that
\begin{eqnarray} \label{adyhye009}
\|\partial_{t}b\|_{L_{t}^{2}L^{2}} \leq C(t,u_{0},b_{0}).
\end{eqnarray}
Noticing that after checking the proof of \eqref{vsadfg11}, it seems impossible to derive \eqref{adfkh1} via only the equation $\eqref{ZEGMHD}_{2}$ or the equation $\eqref{VMHD}_{2}$. To bypass this difficulty, we appeal to explore a combined quantity $\mathcal{L}b-(b\cdot\nabla)u$ which obeys the equation
\begin{align} \label{adyhye012}
\partial_{t}\{\mathcal{L}b-
(b\cdot\nabla)u\}+\mathcal{L}\{\mathcal{L}b-
(b\cdot\nabla)u\}=&b\cdot\nabla\{(u\cdot\nabla)u\}+b\cdot\nabla(\nabla p)-\mathcal{L}\{(u\cdot\nabla)b\}\nonumber\\&-b\cdot\nabla\{(b\cdot\nabla)b\}
-\partial_{t}b\cdot\nabla u.
\end{align}
The derivation of \eqref{adyhye012} can be performed by combining $\eqref{ZEGMHD}_{1}$ and $\eqref{ZEGMHD}_{2}$.
More precisely, we apply $b\cdot\nabla$ to $\eqref{ZEGMHD}_{1}$ to conclude
\begin{align} \label{adyhye010}
\partial_{t}(b\cdot\nabla u)+b\cdot\nabla\{(u\cdot\nabla)u\}+b\cdot\nabla(\nabla p)=b\cdot\nabla\{(b\cdot\nabla)b\}+\partial_{t}b\cdot\nabla u.
\end{align}
Applying $\mathcal{L}$ to $\eqref{ZEGMHD}_{2}$, we thus infer
\begin{align} \label{adyhye011}
\partial_{t}\mathcal{L}b+\mathcal{L}\{(u\cdot\nabla)b\}+\mathcal{L}\{\mathcal{L}b-
(b\cdot\nabla)u\}=0.
\end{align}
As a result, \eqref{adyhye012} is an easy consequence of \eqref{adyhye011} minus \eqref{adyhye010}. We note that applying the direct $L^2$-energy method to \eqref{adyhye012} to derive \eqref{adfkh1} is not workable due to the strong nonlinearity at the right-hand side of \eqref{adyhye012}. To overcome this difficulty, we formulate \eqref{adyhye012} as the following integral form
\begin{align}
\{\mathcal{L}b-
(b\cdot\nabla)u\}(t)=K(t)\ast \{\mathcal{L}b_{0}-
(b_{0}\cdot\nabla)u_{0}\}+\sum_{k=1}^{5}N_{k},\nonumber
\end{align}
where the five terms are given by
$$\quad \quad \ \ N_{1}=\int_{0}^{t}{K(t-\tau)\ast \Big(b\cdot\nabla\{(u\cdot\nabla)u\}\Big)(\tau)\,d\tau},$$
$$ \, N_{2}=\int_{0}^{t}{K(t-\tau)\ast \Big(b\cdot\nabla(\nabla p)\Big)(\tau)\,d\tau},$$
$$ \quad \ \ \ N_{3}=-\int_{0}^{t}{K(t-\tau)\ast \Big(\mathcal{L}\{(u\cdot\nabla)b\}\Big)(\tau)\,d\tau},$$
$$\qquad  \ \ \ \, N_{4}=-\int_{0}^{t}{K(t-\tau)\ast \Big(b\cdot\nabla\{(b\cdot\nabla)b\}\Big)(\tau)\,d\tau}, $$
$$\ N_{5}=-\int_{0}^{t}{K(t-\tau)\ast \Big(\partial_{t}b\cdot\nabla u\Big)(\tau)\,d\tau}.$$
By virtue of the Plancherel Theorem, we notice
\begin{align}
 \|K(t)\ast\{\mathcal{L}b_{0}-
(b_{0}\cdot\nabla)u_{0}\}\|_{L^2}
&\leq \|\widehat{K}(t) \mathcal{F}\{\mathcal{L}b_{0}-
(b_{0}\cdot\nabla)u_{0}\}\|_{L^2} \nonumber\\
&\leq \| \mathcal{F}\{\mathcal{L}b_{0}-
(b_{0}\cdot\nabla)u_{0}\}\|_{L^2} \nonumber\\
&\leq C\|\mathcal{L}b_{0}-
(b_{0}\cdot\nabla)u_{0}\|_{L^2} \nonumber\\
&\leq C, \nonumber
\end{align}
which shows
\begin{align}\label{adyhye013}
\|\{\mathcal{L}b-
(b\cdot\nabla)u\}(t)\|_{L^2}&\leq  C +\sum_{k=1}^{5}\|N_{k}\|_{L^2}.
\end{align}
Next we will claim that
\begin{align}\label{adyhye014}
\sum_{k=1}^{4}\|N_{k}\|_{L_{t}^{\infty}L^{2}} \leq C(t,u_{0},b_{0}).
\end{align}
Actually, by taking $\varepsilon\in (0,1)$, one deduces from \eqref{bnmper8} and the convolution inequality in space variables that
\begin{align}\label{adyhye015}
\|N_{1}\|_{L^{2}}&=\left\|\int_{0}^{t}{K(t-\tau)\ast \Big(b\cdot\nabla\{(u\cdot\nabla)u\}\Big)(\tau)\,d\tau}\right\|_{L^{2}}
\nonumber\\&=\left\|\int_{0}^{t}{\Lambda^{\varepsilon}\nabla K(t-\tau)\ast \Lambda^{-\varepsilon}\Big(b \{(u\cdot\nabla)u\}\Big)(\tau)\,d\tau}\right\|_{L^{2}}
\nonumber\\&\leq C\int_{0}^{t}{\left\|\Lambda^{\varepsilon}\nabla K(t-\tau)\ast \Lambda^{-\varepsilon}\Big(b \{(u\cdot\nabla)u\}\Big)(\tau)\right\|_{L^{2}}\,d\tau}
\nonumber\\&\leq C\int_{0}^{t}{\left\|\Lambda^{\varepsilon}\nabla K(t-\tau)\right\|_{L^{1}} \left\|\Lambda^{-\varepsilon}\Big(b \{(u\cdot\nabla)u\}\Big)(\tau)\right\|_{L^{2}}\,d\tau}
\nonumber\\&\leq C\int_{0}^{t}{\left\|\Lambda^{\varepsilon}\nabla K(t-\tau)\right\|_{L^{1}} \left\| \Big(b \{(u\cdot\nabla)u\}\Big)(\tau)\right\|_{L^{\frac{2}{1+\varepsilon}}}\,d\tau}
\nonumber\\&\leq C\int_{0}^{t}{\left\|\Lambda^{\varepsilon}\nabla K(t-\tau)\right\|_{L^{1}} \|b(\tau)\|_{L^{\infty}} \|u(\tau)\|_{L^{\frac{2}{\varepsilon}}}\|\nabla u(\tau)\|_{L^{2}}\,d\tau}
\nonumber\\&\leq C(t)\int_{0}^{t}{\left\|\Lambda^{\varepsilon}\nabla K(\tau)\right\|_{L^{1}} \,d\tau}\nonumber\\&\leq C(t).
\end{align}
Similarly, one obtains
\begin{align}\label{adyhye016}
\|N_{4}\|_{L^{2}}\leq C(t).
\end{align}
Owing to $\nabla\cdot u=0$, we thus deduce from $\eqref{ZEGMHD}_{1}$ that
$$\nabla p=\frac{\nabla \nabla\cdot(b\cdot\nabla b-u\cdot\nabla u)}{\Delta},$$
which ensures
\begin{align}\label{adyhye017}
\|N_{2}\|_{L^{2}}&=\left\|\int_{0}^{t}{K(t-\tau)\ast \left(b\cdot\nabla\left(\frac{\nabla \nabla\cdot(b\cdot\nabla b-u\cdot\nabla u)}{\Delta}\right)\right)(\tau)\,d\tau}\right\|_{L^{2}}
\nonumber\\&=\left\|\int_{0}^{t}{\Lambda^{\varepsilon}\nabla K(t-\tau)\ast \Lambda^{-\varepsilon}\left(b \frac{\nabla \nabla\cdot(b\cdot\nabla b-u\cdot\nabla u)}{\Delta}\right)(\tau)\,d\tau}\right\|_{L^{2}}
\nonumber\\&\leq C\int_{0}^{t}{\left\|\Lambda^{\varepsilon}\nabla K(t-\tau)\ast \Lambda^{-\varepsilon}\left(b \frac{\nabla \nabla\cdot(b\cdot\nabla b-u\cdot\nabla u)}{\Delta}\right) (\tau)\right\|_{L^{2}}\,d\tau}
\nonumber\\&\leq C\int_{0}^{t}{\left\|\Lambda^{\varepsilon}\nabla K(t-\tau)\right\|_{L^{1}} \left\|\Lambda^{-\varepsilon}\left(b \frac{\nabla \nabla\cdot(b\cdot\nabla b-u\cdot\nabla u)}{\Delta}\right)(\tau)\right\|_{L^{2}}\,d\tau}
\nonumber\\&\leq C\int_{0}^{t}{\left\|\Lambda^{\varepsilon}\nabla K(t-\tau)\right\|_{L^{1}} \left\| b \frac{\nabla \nabla\cdot(b\cdot\nabla b-u\cdot\nabla u)}{\Delta} (\tau)\right\|_{L^{\frac{2}{1+\varepsilon}}}\,d\tau}
\nonumber\\&\leq C\int_{0}^{t}{\left\|\Lambda^{\varepsilon}\nabla K(t-\tau)\right\|_{L^{1}} \|b(\tau)\|_{L^{\infty}}\left\| \frac{\nabla \nabla\cdot(b\cdot\nabla b-u\cdot\nabla u)}{\Delta} (\tau)\right\|_{L^{\frac{2}{1+\varepsilon}}}\,d\tau}
\nonumber\\&\leq C\int_{0}^{t}{\left\|\Lambda^{\varepsilon}\nabla K(t-\tau)\right\|_{L^{1}} \|b(\tau)\|_{L^{\infty}}\left\| (b\cdot\nabla b-u\cdot\nabla u) (\tau)\right\|_{L^{\frac{2}{1+\varepsilon}}}\,d\tau}
\nonumber\\&\leq C\int_{0}^{t}{\left\|\Lambda^{\varepsilon}\nabla K(t-\tau)\right\|_{L^{1}} \|b(\tau)\|_{L^{\infty}} (\|b\|_{L^{\frac{2}{\varepsilon}}}\|\nabla b\|_{L^{2}}+\|u\|_{L^{\frac{2}{\varepsilon}}}\|\nabla u\|_{L^{2}})(\tau)\,d\tau}
\nonumber\\&\leq C(t)\int_{0}^{t}{\left\|\Lambda^{\varepsilon}\nabla K(\tau)\right\|_{L^{1}} \,d\tau}\nonumber\\&\leq C(t).
\end{align}
For the term $N_{3}$, due to $\varepsilon\in (0,1)$, we can also conclude
\begin{align}\label{adyhye018}
\|N_{3}\|_{L^{2}}&=\left\|\int_{0}^{t}{K(t-\tau)\ast \Big(\mathcal{L}\{(u\cdot\nabla)b\}\Big)(\tau)\,d\tau}\right\|_{L^{2}}
\nonumber\\&=\left\|\int_{0}^{t}{\Lambda^{2-\varepsilon} K(t-\tau)\ast \Big(\Lambda^{\varepsilon-2}\mathcal{L}\{(u\cdot\nabla)b\}\Big)(\tau)
\,d\tau}\right\|_{L^{2}}
\nonumber\\&\leq C\int_{0}^{t}{\left\|\Lambda^{2-\varepsilon}  K(t-\tau)\right\|_{L^{1}} \left\|\Big(\Lambda^{\varepsilon-2}\mathcal{L}\{(u\cdot\nabla)b\}\Big)(\tau)\right\|_{L^{2}}\,d\tau}
\nonumber\\&\leq C\int_{0}^{t}{\left\|\Lambda^{2-\varepsilon} K(t-\tau)\right\|_{L^{1}} \left\| \Big(\Lambda^{\varepsilon}\{(u\cdot\nabla)b\}\Big)(\tau)\right\|_{L^{2}}\,d\tau}
\nonumber\\&\leq C\int_{0}^{t}{\left\|\Lambda^{2-\varepsilon} K(t-\tau)\right\|_{L^{1}}  ( \|\Lambda^{\varepsilon}u\|_{L^{\frac{2}{\varepsilon}}}\|\nabla b\|_{L^{\frac{2}{1-\varepsilon}}}+\|u\|_{L^{\frac{2(\varepsilon+1)}{1-\varepsilon}}}
\|\Lambda^{\varepsilon}\nabla b\|_{L^{\frac{\varepsilon+1}{\varepsilon}}}
)(\tau)\,d\tau}
\nonumber\\&\leq C\int_{0}^{t}{\left\|\Lambda^{2-\varepsilon} K(t-\tau)\right\|_{L^{1}}  ( \|\Lambda^{\varepsilon}u\|_{L^{\frac{2}{\varepsilon}}}\|\Lambda^{\varepsilon}j
\|_{L^{2}}+\|u\|_{L^{\frac{2(\varepsilon+1)}{1-\varepsilon}}}
\|\Lambda^{\frac{\varepsilon^{2}+1}{\varepsilon+1}}j\|_{L^{2}}
)(\tau)\,d\tau}
\nonumber\\&\leq C(t)\int_{0}^{t}{\left\|\Lambda^{2-\varepsilon} K(\tau)\right\|_{L^{1}} \,d\tau}\nonumber\\&\leq C(t),
\end{align}
where we have used the fact
$$\left\| \Lambda^{\varepsilon-2}\mathcal{L}f\right\|_{L^{2}}\leq C\left\|\Lambda^{\varepsilon}f\right\|_{L^{2}}.$$
The above estimate can be deduced from the Plancherel Theorem
\begin{align}
\left\| \Lambda^{\varepsilon-2}\mathcal{L}f\right\|_{L^{2}}&=\left(\int_{\mathbb{R}^2}|\xi|^{2\varepsilon-4}
|\widehat{\mathcal{L}f}(\xi)|^{2}\,d\xi\right)^{\frac{1}{2}}\nonumber\\
&=\left(\int_{\mathbb{R}^2}|\xi|^{2\varepsilon-4}\frac{|\xi|^{4}}{g^{2}(\xi)}
|\widehat{f}(\xi)|^{2}\,d\xi\right)^{\frac{1}{2}}\nonumber\\
&\leq \frac{1}{C_{0}}\left(\int_{\mathbb{R}^2}|\xi|^{2\varepsilon}
|\widehat{f}(\xi)|^{2}\,d\xi\right)^{\frac{1}{2}}\nonumber\\
&\leq \frac{1}{C_{0}}\left\|\Lambda^{\varepsilon}f\right\|_{L^{2}}.\nonumber
\end{align}
Putting the estimates \eqref{adyhye015}, \eqref{adyhye016}, \eqref{adyhye017} and \eqref{adyhye018} together yields \eqref{adyhye014}.
For the last term, at this stage, we are only able to show
\begin{eqnarray}\label{adyhye019}
\|N_{5}\|_{L_{t}^{q}L^{2}} \leq C(t,u_{0},b_{0})
\end{eqnarray}
for any $q<\infty$. As a matter of fact, we also conclude
\begin{align} \label{adyhye020}
\|N_{5}\|_{L^{2}}&=\left\|\int_{0}^{t}{K(t-\tau)\ast \Big(\partial_{t}b\cdot\nabla u\Big)(\tau)\,d\tau}\right\|_{L^{2}}
\nonumber\\&\leq C\int_{0}^{t}{\left\|K(t-\tau)\right\|_{L^{2}} \left\|\partial_{t}b\cdot\nabla u(\tau)\right\|_{L^{1}}\,d\tau}
\nonumber\\&\leq C\int_{0}^{t}{\left\|K(t-\tau)\right\|_{L^{2}} \left\|\partial_{t}b(\tau)\right\|_{L^{2}}\left\|\nabla u(\tau)\right\|_{L^{2}}\,d\tau}\nonumber\\&\leq C(t)\int_{0}^{t}{\left\|K(t-\tau)\right\|_{L^{2}} \left\|\partial_{t}b(\tau)\right\|_{L^{2}}\,d\tau}.
\end{align}
Applying the convolution inequality in time variables and using \eqref{adyhye020} as well as \eqref{adyhye009}, we readily obtain that
\begin{align}
\|N_{5}\|_{L_{t}^{q}L^{2}}\leq C\|K(t)\|_{L_{t}^{\frac{2q}{q+2}}L^{2}}\|\partial_{t}b(t)\|_{L_{t}^{2}L^{2}}\leq C(t),\nonumber
\end{align}
where we have used the following fact due to \eqref{bnmper8}
$$\|K(t)\|_{L_{t}^{\frac{2q}{q+2}}L^{2}}\leq C(t).$$
Plugging \eqref{adyhye014} and \eqref{adyhye019} into \eqref{adyhye013} implies for any $q\in[2,\infty)$
\begin{align}
\|\{\mathcal{L}b-
(b\cdot\nabla)u\}(t)\|_{L_{t}^{q}L^2} \leq  C(t).\nonumber
\end{align}
This enables us to show
\begin{align}
\|\mathcal{L}b(t)\|_{L_{t}^{q}L^2}&\leq \|\{\mathcal{L}b-
(b\cdot\nabla)u\}(t)\|_{L_{t}^{q}L^2}+\|(b\cdot\nabla)u(t)\|_{L_{t}^{q}L^2} \nonumber\\&
 \leq \|\{\mathcal{L}b-
(b\cdot\nabla)u\}(t)\|_{L_{t}^{q}L^2}+\|b(t)\|_{L_{t}^{\infty}L^{\infty}}
\|\nabla u(t)\|_{L_{t}^{q}L^2}
\nonumber\\&\leq  C(t),\nonumber
\end{align}
which is the desired bound \eqref{adyhye004}. With \eqref{adyhye004} in hand, we are in a position to prove \eqref{adfkh1}.
Actually, coming back to \eqref{adyhye020}, we have
\begin{align}
\|N_{5}\|_{L^{2}}\leq C(t)\int_{0}^{t}{\left\|K(t-\tau)\right\|_{L^{2}} \left\|\partial_{t}b(\tau)\right\|_{L^{2}}\,d\tau}.\nonumber
\end{align}
This along with the convolution inequality in time variables yields
\begin{align}\label{adyhye021}
\|N_{5}\|_{L_{t}^{\infty}L^{2}} \leq C\|K(t)\|_{L_{t}^{\frac{q}{q-1}}L^{2}}\|\partial_{t}b(t)\|_{L_{t}^{q}L^{2}}
 \leq C(t),
\end{align}
where we have used the following fact due to \eqref{bnmper8} again
$$\|K(t)\|_{L_{t}^{\frac{q}{q-1}}L^{2}}\leq C(t).$$
Putting \eqref{adyhye014} and \eqref{adyhye021} into \eqref{adyhye013} gives the bound
\begin{align}
\|\{\mathcal{L}b-
(b\cdot\nabla)u\}(t)\|_{L_{t}^{\infty}L^2} \leq  C(t).\nonumber
\end{align}
This also allows us to derive
\begin{align}
\|\mathcal{L}b(t)\|_{L_{t}^{\infty}L^2}&\leq \|\{\mathcal{L}b-
(b\cdot\nabla)u\}(t)\|_{L_{t}^{\infty}L^2}+\|(b\cdot\nabla)u(t)\|_{L_{t}^{\infty}L^2} \nonumber\\&
 \leq \|\{\mathcal{L}b-
(b\cdot\nabla)u\}(t)\|_{L_{t}^{\infty}L^2}+\|b(t)\|_{L_{t}^{\infty}L^{\infty}}
\|\nabla u(t)\|_{L_{t}^{\infty}L^2}
\nonumber\\&\leq  C(t),\nonumber
\end{align}
which is \eqref{adfkh1}. Therefore, we complete the proof of Lemma \ref{vsadlea11}.
\end{proof}

\vskip .1in
For the combined quantity $\mathcal{L}b-(b\cdot\nabla)u$ itself, we are able to show that it actually belongs to $L_{t}^{\infty}H^{r}$ for any $r\in [0,1)$. Precisely, it reads as follows.
\begin{lemma}  \label{vsg82}
Let $(u_{0},b_{0})$ satisfy the conditions stated in Theorem \ref{Th2}, then it holds for any $r\in[0,1)$
\begin{eqnarray}\label{adyhye022}
\|\Lambda^{r}\{\mathcal{L}b-
(b\cdot\nabla)u\}(t)\|_{L^{2}} \leq C(t,u_{0},b_{0}).
\end{eqnarray}
\end{lemma}

\begin{proof}
According to \eqref{adyhye012}, we may get
\begin{align}
\Lambda^{r}\{\mathcal{L}b-
(b\cdot\nabla)u\}(t)=K(t)\ast \Lambda^{r}\{\mathcal{L}b_{0}-
(b_{0}\cdot\nabla)u_{0}\}+\sum_{k=1}^{5}\widetilde{N}_{k},\nonumber
\end{align}
where the five terms are given by
$$\quad \quad \ \ \widetilde{N}_{1}=\int_{0}^{t}{K(t-\tau)\ast \Lambda^{r}\Big(b\cdot\nabla\{(u\cdot\nabla)u\}\Big)(\tau)\,d\tau},$$
$$ \, \widetilde{N}_{2}=\int_{0}^{t}{K(t-\tau)\ast \Lambda^{r}\Big(b\cdot\nabla(\nabla p)\Big)(\tau)\,d\tau},$$
$$ \quad \ \ \ \widetilde{N}_{3}=-\int_{0}^{t}{K(t-\tau)\ast \Lambda^{r}\Big(\mathcal{L}\{(u\cdot\nabla)b\}\Big)(\tau)\,d\tau},$$
$$\qquad  \ \ \ \, \widetilde{N}_{4}=-\int_{0}^{t}{K(t-\tau)\ast \Lambda^{r}\Big(b\cdot\nabla\{(b\cdot\nabla)b\}\Big)(\tau)\,d\tau}, $$
$$\ \widetilde{N}_{5}=-\int_{0}^{t}{K(t-\tau)\ast \Lambda^{r}\Big(\partial_{t}b\cdot\nabla u\Big)(\tau)\,d\tau}.$$
We thus have
\begin{align}\label{adyhye023}
\|\Lambda^{r}\{\mathcal{L}b-
(b\cdot\nabla)u\}(t)\|_{L^2}&\leq \|K(t)\ast \Lambda^{r}\{\mathcal{L}b_{0}-
(b_{0}\cdot\nabla)u_{0}\}\|_{L^{2}} +\sum_{k=1}^{5}\|\widetilde{N}_{k}\|_{L^2}.
\end{align}
Due to $\Lambda^{r}\{\mathcal{L}b_{0}-
(b_{0}\cdot\nabla)u_{0}\}\in L^{2}(\mathbb{R}^{2})$, according to \eqref{ydvba11}, we get
$$\|K(t)\ast \Lambda^{r}\{\mathcal{L}b_{0}-
(b_{0}\cdot\nabla)u_{0}\}\|_{L^{2}}\leq C.$$
Obviously, it follows from $\nabla\cdot u=\nabla\cdot b=0$ that
\begin{align}
 b\cdot\nabla\{(u\cdot\nabla)u\}&= b_{i}\partial_{i}(u_{j}\partial_{j}u_{k})\nonumber\\
 &= \partial_{i}\{b_{i}(u_{j}\partial_{j}u_{k})\}\nonumber\\
 &= \partial_{i}\{b_{i}\partial_{j}(u_{j}u_{k})\}\nonumber\\
 &= \partial_{i}\partial_{j}\{b_{i}(u_{j}u_{k})\}-\partial_{i}
 \{\partial_{j}b_{i}(u_{j}u_{k})\},\nonumber
\end{align}
where we have adopted the Einstein summation convention.
In what follows, taking $\varepsilon\in (0,1-r)$, we get from \eqref{bnmper8} and the convolution inequality in space variables
\begin{align}
\|\widetilde{N}_{1}\|_{L^{2}}&=\left\|\int_{0}^{t}{K(t-\tau)\ast \Lambda^{r} \Big(\partial_{i}\partial_{j}\{b_{i}(u_{j}u_{k})\}-\partial_{i}\{\partial_{j}b_{i}
(u_{j}u_{k})\}\Big)(\tau)\,d\tau}\right\|_{L^{2}}
\nonumber\\&\leq\left\|\int_{0}^{t}{K(t-\tau)\ast \Lambda^{r} \Big(\partial_{i}\partial_{j}\{b_{i}(u_{j}u_{k})\}\Big)(\tau)\,d\tau}\right\|_{L^{2}}
\nonumber\\&\quad+
\left\|\int_{0}^{t}{K(t-\tau)\ast \Lambda^{r} \Big(\partial_{i}\{\partial_{j}b_{i}
(u_{j}u_{k})\}\Big)(\tau)\,d\tau}\right\|_{L^{2}}
\nonumber\\&\leq\left\|\int_{0}^{t}{\Lambda^{r+1+\varepsilon}K(t-\tau)\ast \Lambda^{-1-\varepsilon} \Big(\partial_{i}\partial_{j}\{b_{i}(u_{j}u_{k})\}\Big)(\tau)\,d\tau}\right\|_{L^{2}}
\nonumber\\&\quad+
\left\|\int_{0}^{t}{\Lambda^{r+1+\varepsilon}K(t-\tau)\ast \Lambda^{-1-\varepsilon} \Big(\partial_{i}\{\partial_{j}b_{i}
(u_{j}u_{k})\}\Big)(\tau)\,d\tau}\right\|_{L^{2}}
\nonumber\\&\leq \int_{0}^{t}{\left\|\Lambda^{r+1+\varepsilon} K(t-\tau)\ast \Lambda^{-1-\varepsilon} \Big(\partial_{i}\partial_{j}\{b_{i}(u_{j}u_{k})\}\Big)(\tau)\right\|_{L^{2}}\,d\tau}
\nonumber\\&\quad+
\int_{0}^{t}{\left\|\Lambda^{r+1+\varepsilon}K(t-\tau)\ast \Lambda^{-1-\varepsilon} \Big(\partial_{i}\{\partial_{j}b_{i}
(u_{j}u_{k})\}\Big)(\tau)\right\|_{L^{2}}\,d\tau}
\nonumber\\&\leq C\int_{0}^{t}{\left\|\Lambda^{r+1+\varepsilon} K(t-\tau)\right\|_{L^{1}} \left\|\Lambda^{1-\varepsilon}(b_{i}u_{j}u_{k})(\tau)\right\|_{L^{2}}\,d\tau}
\nonumber\\&\quad+C\int_{0}^{t}{\left\|\Lambda^{r+1+\varepsilon} K(t-\tau)\right\|_{L^{1}} \left\|\Lambda^{-\varepsilon}(\partial_{j}b_{i}
u_{j}u_{k})(\tau)\right\|_{L^{2}}\,d\tau}
\nonumber\\&\leq C\int_{0}^{t}{\left\|\Lambda^{r+1+\varepsilon} K(t-\tau)\right\|_{L^{1}}  (\|\Lambda^{1-\varepsilon}b\|_{L^{\frac{2}{1-\varepsilon}}}\|u\|_{L^{\frac{4}{\varepsilon}}}^{2}
+\|\Lambda^{1-\varepsilon}u\|_{L^{\frac{2}{1-\varepsilon}}}\|u\|_{L^{\frac{4}{\varepsilon}}}
\|b\|_{L^{\frac{4}{\varepsilon}}})\,d\tau}
\nonumber\\&\quad+C\int_{0}^{t}{\left\|\Lambda^{r+1+\varepsilon} K(t-\tau)\right\|_{L^{1}} \left\|(\partial_{j}b_{i}
u_{j}u_{k})(\tau)\right\|_{L^{\frac{2}{1+\varepsilon}}}\,d\tau}
\nonumber\\&\leq C\int_{0}^{t}{\left\|\Lambda^{r+1+\varepsilon} K(t-\tau)\right\|_{L^{1}}  (\|\nabla b\|_{L^{2}}\|u\|_{L^{\frac{4}{\varepsilon}}}^{2}
+\|\nabla u\|_{L^{2}}\|u\|_{L^{\frac{4}{\varepsilon}}}
\|b\|_{L^{\frac{4}{\varepsilon}}})(\tau)\,d\tau}
\nonumber\\&\quad+C\int_{0}^{t}{\left\|\Lambda^{r+1+\varepsilon} K(t-\tau)\right\|_{L^{1}} \|\nabla b\|_{L^{2}}\|u\|_{L^{\frac{4}{\varepsilon}}}^{2}(\tau)\,d\tau}
\nonumber\\&\leq C(t)\int_{0}^{t}{\left\|\Lambda^{r+1+\varepsilon}  K(\tau)\right\|_{L^{1}} \,d\tau}\nonumber\\&\leq C(t).\nonumber
\end{align}
Similarly, we are able to show
\begin{align}
\|\widetilde{N}_{2}\|_{L^{2}}&=\left\|\int_{0}^{t}{K(t-\tau)\ast \Lambda^{r}\left(b\cdot\nabla\left(\frac{\nabla \nabla\cdot(b\cdot\nabla b-u\cdot\nabla u)}{\Delta}\right)\right)(\tau)\,d\tau}\right\|_{L^{2}}
\nonumber\\&=\left\|\int_{0}^{t}{\Lambda^{r+1+\varepsilon}K(t-\tau)\ast \Lambda^{-1-\varepsilon}\nabla\left(b \frac{\nabla \nabla\cdot(b\cdot\nabla b-u\cdot\nabla u)}{\Delta}\right)(\tau)\,d\tau}\right\|_{L^{2}}
\nonumber\\&\leq C\int_{0}^{t}{\left\|\Lambda^{r+1+\varepsilon} K(t-\tau)\ast \Lambda^{-1-\varepsilon}\nabla\left(b \frac{\nabla \nabla\cdot(b\cdot\nabla b-u\cdot\nabla u)}{\Delta}\right) (\tau)\right\|_{L^{2}}\,d\tau}
\nonumber\\&\leq C\int_{0}^{t}{\left\|\Lambda^{r+1+\varepsilon} K(t-\tau)\right\|_{L^{1}} \left\|\Lambda^{-\varepsilon}\left(b \frac{\nabla \nabla\cdot(b\cdot\nabla b-u\cdot\nabla u)}{\Delta}\right)(\tau)\right\|_{L^{2}}\,d\tau}
\nonumber\\&\leq C\int_{0}^{t}{\left\|\Lambda^{r+1+\varepsilon}K(t-\tau)\right\|_{L^{1}} \left\| b \frac{\nabla \nabla\cdot(b\cdot\nabla b-u\cdot\nabla u)}{\Delta} (\tau)\right\|_{L^{\frac{2}{1+\varepsilon}}}\,d\tau}
\nonumber\\&\leq C\int_{0}^{t}{\left\|\Lambda^{r+1+\varepsilon} K(t-\tau)\right\|_{L^{1}} \|b(\tau)\|_{L^{\infty}}\left\| (b\cdot\nabla b-u\cdot\nabla u) (\tau)\right\|_{L^{\frac{2}{1+\varepsilon}}}\,d\tau}
\nonumber\\&\leq C\int_{0}^{t}{\left\|\Lambda^{r+1+\varepsilon} K(t-\tau)\right\|_{L^{1}} \|b(\tau)\|_{L^{\infty}} (\|b\|_{L^{\frac{2}{\varepsilon}}}\|\nabla b\|_{L^{2}}+\|u\|_{L^{\frac{2}{\varepsilon}}}\|\nabla u\|_{L^{2}})(\tau)\,d\tau}
\nonumber\\&\leq C(t)\int_{0}^{t}{\left\|\Lambda^{r+1+\varepsilon} K(\tau)\right\|_{L^{1}} \,d\tau}\nonumber\\&\leq C(t),\nonumber
\end{align}
which also implies
$$\|\widetilde{N}_{4}\|_{L^{2}}\leq C(t).$$
Moreover, we obtain
\begin{align}
\|\widetilde{N}_{3}\|_{L^{2}}&=\left\|\int_{0}^{t}{K(t-\tau)\ast \Lambda^{r} \Big(\mathcal{L}\{(u\cdot\nabla)b\}\Big)(\tau)\,d\tau}\right\|_{L^{2}}
\nonumber\\&=\left\|\int_{0}^{t}{\Lambda^{r+1+\varepsilon} K(t-\tau)\ast \Big(\Lambda^{-\varepsilon-1}\mathcal{L}\{(u\cdot\nabla)b\}\Big)(\tau)
\,d\tau}\right\|_{L^{2}}
\nonumber\\&\leq C\int_{0}^{t}{\left\|\Lambda^{r+1+\varepsilon}  K(t-\tau)\right\|_{L^{1}} \left\|\Big(\Lambda^{1-\varepsilon-2}\mathcal{L}\{(u\cdot\nabla)b\}\Big)(\tau)\right\|_{L^{2}}\,d\tau}
\nonumber\\&\leq C\int_{0}^{t}{\left\|\Lambda^{r+1+\varepsilon} K(t-\tau)\right\|_{L^{1}} \left\| \Big(\Lambda^{1-\varepsilon}\{(u\cdot\nabla)b\}\Big)(\tau)\right\|_{L^{2}}\,d\tau}
\nonumber\\&\leq C\int_{0}^{t}{\left\|\Lambda^{r+1+\varepsilon}K(t-\tau)\right\|_{L^{1}}
C(\tau)\,d\tau}
\nonumber\\&\leq C(t)\int_{0}^{t}{\left\|\Lambda^{r+1+\varepsilon} K(\tau)\right\|_{L^{1}} \,d\tau}\nonumber\\&\leq C(t),\nonumber
\end{align}
where we have used the following estimate
\begin{align}
\left\| \Lambda^{1-\varepsilon}\{(u\cdot\nabla)b\}(t)\right\|_{L^{2}}&\leq C\|\Lambda^{1-\varepsilon}u\|_{L^{\frac{2}{1-\varepsilon}}}\|\nabla b\|_{L^{\frac{2}{\varepsilon}}}+C\|u\|_{L^{\frac{2(2-\varepsilon)}{\varepsilon}}}
\|\Lambda^{1-\varepsilon}\nabla b\|_{L^{\frac{2-\varepsilon}{1-\varepsilon}}}\nonumber\\&\leq
C\|\omega\|_{L^{2}}
\|\Lambda^{1-\varepsilon}j\|_{L^{2}}
+C\|u\|_{L^{\frac{2(2-\varepsilon)}{\varepsilon}}}
\|\Lambda^{\frac{2-2\varepsilon+\varepsilon^{2}}{2-\varepsilon}}j\|_{L^{2}}
\nonumber\\&\leq
C(t).\nonumber
\end{align}
We deduce from \eqref{adyhye005}, \eqref{adfkh1} and \eqref{adyhye007} that
\begin{align}\label{adyhye023}
\|\partial_{t}b\|_{L_{t}^{\infty}L^{2}} \leq \|\mathcal{L}b\|_{L_{t}^{\infty}L^{2}}+\|f\|_{L_{t}^{\infty}L^{2}}\leq
C(t).
\end{align}
According to \eqref{bnmper8} and \eqref{adyhye023}, we finally get
\begin{align}
\|N_{5}\|_{L^{2}}&=\left\|\int_{0}^{t}\Lambda^{r}{K(t-\tau)\ast  \Big(\partial_{t}b\cdot\nabla u\Big)(\tau)\,d\tau}\right\|_{L^{2}}
\nonumber\\&\leq C\int_{0}^{t}{\left\|\Lambda^{r} K(t-\tau)\right\|_{L^{2}} \left\|\partial_{t}b\cdot\nabla u(\tau)\right\|_{L^{1}}\,d\tau}
\nonumber\\&\leq C\int_{0}^{t}{\left\|\Lambda^{r} K(t-\tau)\right\|_{L^{2}} \left\|\partial_{t}b(\tau)\right\|_{L^{2}}\left\|\nabla u(\tau)\right\|_{L^{2}}\,d\tau}\nonumber\\&\leq C(t)\int_{0}^{t}{\left\|\Lambda^{r} K(t-\tau)\right\|_{L^{2}}\,d\tau}
\nonumber\\&\leq C(t).\nonumber
\end{align}
Putting the above five estimates into \eqref{adyhye023}, we arrive at the desired bound \eqref{adyhye022}. Therefore, we conclude the proof of Lemma \ref{vsg82}.
\end{proof}

\vskip .2in
\textbf{Acknowledgements.}
Ye was supported by the Qing Lan Project of Jiangsu Province. Zhao was partially supported by the National Natural Science Foundation of China (No. 11901165, No. 11971446).


\vskip .3in
\begin{thebibliography}{00} \frenchspacing
\bibitem{Agelas}
L. Agelas, \emph{Global regularity for logarithmically critical 2D MHD equations with zero viscosity}, Monatsh. Math.  \textbf{181}  (2016), 245--266.

\bibitem{BCD}
H. Bahouri, J.-Y. Chemin, R. Danchin: \emph{Fourier Analysis and Nonlinear
Partial Differential Equations}, Grundlehren der mathematischen
Wissenschaften, 343, Springer (2011).


\bibitem{CW2011}
C. Cao, J. Wu, \emph{Global regularity for the 2D MHD equations with mixed
partial dissipation and magnetic diffusion}, Adv. Math. \textbf{226} (2011),
1803--1822.

\bibitem{CWYSiam14}
C. Cao, J. Wu, B. Yuan, \emph{The 2D incompressible magnetohydrodynamics
equations with only magnetic diffusion}, SIAM J. Math. Anal. \textbf{{46}} (2014), 588--602.

\bibitem{Chamorrolr}
D. Chamorro, P.G. Lemarie-Rieusset, \emph{Quasi-geostrophic equation, nonlinear Bernstein inequalities and
$\alpha$-stable processes}, Rev. Mat. Iberoam. \textbf{28} (2012), 1109--1122.

\bibitem{Davidson01}
P.A. Davidson, \emph{An Introduction to Magnetohydrodynamics}, Cambridge University Press, Cambridge, England, 2001.

\bibitem{FNZ14MM}
J. Fan, H. Malaikah, S. Monaquel, G. Nakamura, Y. Zhou, \emph{Global Cauchy problem of 2D generalized MHD equations}, Monatsh. Math. \textbf{{175}} (2014), 127--131.

\bibitem{JZ31114}
Q. Jiu, J. Zhao, \emph{A remark on global regularity of 2D generalized
magnetohydrodynamic equations}, J. Math. Anal. Appl. \textbf{{412}} (2014),
478--484.

\bibitem{JZ31115}
Q. Jiu, J. Zhao, \emph{Global regularity of 2D generalized MHD equations with magnetic diffusion}, Z. Angew. Math. Phys. \textbf{ {66} }(2015), 677--687.

\bibitem{kaPonce}
T. Kato, G. Ponce, \emph{Commutator estimates and the Euler and the Navier-Stokes equations},
Comm. Pure Appl. Math. \textbf{41} (1988), 891--907.

\bibitem{LZ}
Z. Lei, Y. Zhou, \emph{BKM's criterion and global weak solutions for
magnetohydrodynamics with zero viscosity}, Discrete Contin. Dyn.
Syst. \textbf{25} (2009), 575--583.

\bibitem{PF}
E. Priest, T. Forbes, \emph{Magnetic reconnection, MHD theory and
Applications}, Cambridge University Press, Cambridge, 2000.

\bibitem{ST}
M. Sermange, R. Temam, \emph{Some mathematical questions related to the
MHD equations}, Comm. Pure Appl. Math. \textbf{36} (1983), 635--664.

\bibitem{TTao}
T. Tao, \emph{Global regularity for a logarithmically supercritical hyperdissipative Navier-Stokes equation}, Anal. PDE \textbf{2} (2009), 361--366.


\bibitem{TYZ}
C. V. Tran, X. Yu, Z. Zhai, \emph{Note on solution regularity of the
generalized magnetohydrodynamic equations with partial dissipation},
Nonlinear Anal. \textbf{85} (2013), 43--51.

\bibitem{TYZ113}
C. V. Tran, X. Yu, Z. Zhai, \emph{On global regularity of 2D generalized
magnetodydrodynamics equations}, J. Differential. Equations \textbf{{254}}
(2013), 4194--4216.

\bibitem{Wu2003}
J. Wu, \emph{The generalized MHD equations}, J. Differential. Equations, \textbf{195} (2003), 284--312.

\bibitem{Wu2011}
J. Wu, \emph{Global regularity for a class of generalized
magnetohydrodynamic equations}, J. Math. Fluid Mech, \textbf{13} (2011), 295--305.

\bibitem{Y3efg5}
K. Yamazaki, \emph{On the global regularity of two-dimensional generalized
magnetohydrodynamics system}, J. Math. Anal. Appl. \textbf{{416}} (2014), 99--111.

\bibitem{Yasd14a}
K. Yamazaki, \emph{Global regularity of the logarithmically supercritical MHD system with zero diffusivity}, Appl. Math. Lett. \textbf{29} (2014), 46--51.

\bibitem{Yamazaki18}
K. Yamazaki, \emph{Global regularity of logarithmically supercritical MHD system with improved logarithmic powers}, Dyn. Partial Differ. Equ. \textbf{15} (2018), 147--173.

\bibitem{Yejee18}
Z. Ye, \emph{Remark on the global regularity of 2D MHD equations with almost Laplacian magnetic diffusion}, J. Evol. Equ. \textbf{18} (2018), 821--844.


\bibitem{YX2014NA}
Z. Ye, X. Xu, \emph{Global regularity of the two-dimensional
incompressible generalized magnetohydrodynamics system}, Nonlinear
Anal. \textbf{{100}} (2014), 86--96.


\bibitem{Yeaam18}
Z. Ye, \emph{Some new regularity criteria for the 2D
Euler-Boussinesq equations via the temperature}, Acta Appl Math \textbf{157} (2018), 141--169.

\bibitem{YB2014JMAA}
B. Yuan, L. Bai, \emph{Remarks on global regularity of 2D generalized MHD
equations}, J. Math. Anal. Appl. \textbf{{413}} (2014), 633--640.

\bibitem{YZhao16}
B. Yuan, J. Zhao,
\emph{Global regularity of 2D almost resistive MHD Equations},Nonlinear Anal. Real World Appl. \textbf{41} (2018), 53--65.


\bibitem{Zhao22}
J. Zhao, \emph{Global regularity for solutions to 2D generalized MHD equations with multiple exponential upper bound uniformly in time}, J. Math. Anal. Appl. \textbf{514} (2022), no. 1, Paper No. 126306.

\end{thebibliography}
\end{document}